\documentclass[12pt,thmsa, reqno]{amsart}

\usepackage{amsmath} 
\usepackage{amssymb, euscript, amsthm, amsfonts} 

\usepackage{graphicx} 
\usepackage{multicol}
\usepackage{amscd,mathrsfs}
\usepackage{caption}

\def\th{\theta}

\def\Cal{\mathcal}

\def\K{{\Cal K}}
\def\I{{\Cal I}}

\def\T{{\Cal T}}

\def\G{\mathcal{G}}

\def\Q{\mathcal{Q}}

\def\f0{f_0}
\def\Fc0{\varphi_0}
\def\rn{\bbr^n}

\def\I_k {I_{-}^{k/2}}
\def\I+k {I_{+}^{k/2}}

\def\bbr{{\Bbb R}}

\def\bbn{{\Bbb N}}
\def\bbh{{\Bbb H}}

\def\bbc{{\Bbb C}}
\def\bbz{{\Bbb Z}}

\def\bbe{{\Bbb E}}
\def \hn{{\bbh^n}}

\def\cosh{{\hbox{\rm cosh}}}
\def\sinh{{\hbox{\rm sinh}}}

\def\const{{\hbox{\rm const}}}
\def\cos{{\hbox{\rm cos}}}

\def\ch{{\hbox{\rm cosh}}}
\def\sh{{\hbox{\rm sinh}}}
\def\tanh{{\hbox{\rm tanh}}}

\def \hns {\overset {*}{\bbh}{}^n}

\def\part{\partial}
\def\intl{\int\limits}
\def\b{\beta}

\def\Gam{\Gamma}
\def\Om{\Omega}
\def\a{\alpha}
\def\om{\omega}

\def\del{\delta}
\def\vp{\varphi}

\def\gam{\gamma}

\def\sig{\sigma}
\def\lam{\lambda}

\def\e{\varepsilon}
\def\t{\tau}

\def\sn{S^{n-1}}

\font\frak=eufm10

\def\fr#1{\hbox{\frak #1}}

\def\frR{\fr{R}}        
\def\frS{\fr{S}}

\def\const{{\hbox{\rm const}}}
\def\cos{{\hbox{\rm cos}}}

\def\part{\partial}
\def\intl{\int\limits}
\def\b{\beta}

\def\Gam{\Gamma}
\def\Om{\Omega}
\def\a{\alpha}

\def\hn{\bbh^n}

\newtheorem{theorem}{Theorem}[section]
\newtheorem{lemma}[theorem]{Lemma}

\newtheorem{corollary}[theorem]{Corollary}
\newtheorem{proposition}[theorem]{Proposition}

\theoremstyle{remark}
\newtheorem{remark}[theorem]{Remark}

\newtheorem{example}[theorem]{Example}

\numberwithin{equation}{section}

\newcommand{\be}{\begin{equation}}
\newcommand{\ee}{\end{equation}}

\newcommand{\bea}{\begin{eqnarray}}
\newcommand{\eea}{\end{eqnarray}}
\newcommand{\Bea}{\begin{eqnarray*}}
\newcommand{\Eea}{\end{eqnarray*}}

\def\sideremark#1{\ifvmode\leavevmode\fi\vadjust{\vbox to0pt{\vss
 \hbox to 0pt{\hskip\hsize\hskip1em
\vbox{\hsize2cm\tiny\raggedright\pretolerance10000
 \noindent #1\hfill}\hss}\vbox to8pt{\vfil}\vss}}}%

\begin{document}

\title[Gegenbauer-Chebyshev  Integrals and  Radon Transforms]
{Gegenbauer-Chebyshev  Integrals and  Radon Transforms}

\author{B. Rubin }
\address{Department of Mathematics, Louisiana State University, Baton Rouge,
Louisiana 70803, USA}
\email{borisr@math.lsu.edu}






\begin{abstract}

We suggest  new  modifications of   Helgason's support theorems  and  descriptions of the kernels for several projectively equivalent transforms of integral geometry.  The paper deals with the  hyperplane Radon transform and its dual, the totally geodesic transforms on the sphere and the hyperbolic space, the spherical slice transform, and the spherical mean transform for spheres through the origin.
The assumptions for functions are   formulated in integral terms.   The proofs rely on the properties of the Gegenbauer-Chebyshev  integrals which    generalize Abel type fractional integrals on the positive half-line.
 \end{abstract}

\maketitle

\section{Introduction}
\setcounter{equation}{0}

The Radon transform assigns to a function $f$ on $\rn$  its integrals $(Rf)(\xi)=\int_\xi f$ over  hyperplanes $\xi$ in $\rn$. Diverse modifications of this transform
   are widely used in image reconstruction problems \cite{AKK, Ku14, Na1}. The related Gegenbauer-Chebyshev  integrals (see (\ref{4gt6a}), (\ref {4gt6a1}),  (\ref{4gt6ale}),  (\ref{4gt6a1le}) below)  which    generalize Abel type operators of fractional integration play an important role in the study of Radon transforms. Information about these integrals can be found, e.g., in   \cite {vBer, vBE, SKM}.
According to Ludwig \cite[p. 50]{Lud}, who referred to the private communication by  L. Sarason, the connection between the Gegenbauer-Chebyshev integrals and the Radon transform   on $\bbr^n$ was  known to  Lax and Phillips \cite {LP62}.  In the case $n=2$, it was independently discovered by Cormack \cite{Co63}; see also subsequent works  by    Cnops \cite{Cno}, Cormack and Quinto \cite{CoQ}, Deans \cite[Chapter 7] {Dea},  Helgason \cite[Chapter I, Section 2]{H08},   Natterer \cite[p. 25]{Na1}.

 A simple unilateral structure
 of the Gegenbauer-Chebyshev  integrals    can be used to retrieve information about the support of a function  from the knowledge of the support of its Radon transform.  The last observation is closely related to the celebrated
  Helgason's support theorem which states the following.

   \begin{theorem} \label{Helgason's}
   If   $(Rf)(\xi)=0$ for all  hyperplanes $\xi$ that do not meet a ball of radius $a>0$, then $f(x)=0$ for all $x$ outside of that ball provided that
  \be\label{on tra} f\!\in \!C(\rn) \quad \mbox{\rm and} \quad  \sup_x |x|^m |f(x)| <\infty \quad \forall  \, m>0.\ee
  \end{theorem}

This  result which extends to arbitrary convex sets in $\rn$ dates back  to Helgason's 1963 address \cite {H63}.  It was mentioned in \cite[p. 438] {H64} and presented with detailed proof in \cite{H65};  see also \cite[ p. 10]{H11}.
  A short proof of   Theorem \ref{Helgason's}  for compactly supported functions was suggested by Strichartz \cite{Str82}. Strichartz's proof was modified  by Boman and   Lindskog \cite{Bom09} for Radon transforms of measures.

 \vskip 0.2 truecm

\noindent{\bf Question:} {\it Can the assumptions in (\ref{on tra}) be weakened}?

 \vskip 0.2 truecm

This question which is intimately connected with the structure of the kernel of the operator  $R$ is the main concern  of the present paper.
 The basic tool   is the Gegenbauer-Chebyshev  integrals for which we establish new facts.
Unlike the aforementioned  publications  where the functions  are continuous and rapidly decreasing, we deal with arbitrary locally integrable functions satisfying  certain integral conditions.

 \vskip 0.2truecm

{\bf Plan of the paper, main results,  and  comments.}

Section 2 contains necessary preliminaries.  Section 3 is devoted to the Gegenbauer-Chebyshev  integrals.
In Section 4 we revise  known facts about the action of the Radon transform and its dual on the subspaces generated by spherical harmonics and prove  the following generalization of Theorem \ref{Helgason's}.

 \begin{theorem} \label{Helgason's2} Let  $B_a^-=\{x \in \rn: |x|>a\}$, $a>0$.
If $f(x)=0$  for almost all  $x\in B_a^-$, then   $(Rf)(\xi)=0$ for almost all hyperplanes $\xi$ in this domain.   The converse is true if  \be\label{osllsgMKL} \intl_{B_a^-} |f(x)|\,|x|^m\, dx <\infty\quad \forall  \, m>0\ee
and  fails otherwise.
\end{theorem}

A natural addition to  Theorem \ref{Helgason's2} is the following statement describing the kernel  of the operator $R$.   This description  is given in terms of the Fourier-Laplace coefficients.  Equivalent statements in any  topological space  containing appropriate linear combinations of spherical harmonics can be obtained by taking  closure in the corresponding topology.

To state the result, let $\{Y_{m,\mu}\}$ be an  orthonormal basis of real-valued spherical harmonics in $L^2(S^{n-1})$; see, e.g., \cite{Mu}. Here
 $m=0,1,2, \ldots,$ and $ \mu=1,2, \ldots d_n (m)$, where
\be\label{kWSQRT23} d_n(m) =(n+2m-2)\,
\frac{(n+m-3)!}{m! \, (n-2)!}\ee
is the dimension of the subspace  of spherical harmonics of degree $m$. Given a function $f$ on $\rn$, the corresponding Fourier-Laplace coefficients are defined by
\[ f_{m,\mu} (r)=\intl_{\sn} f(r\th) Y_{m,\mu} (\th)\, d\th.\]

\begin{theorem} \label{azw1a223}
Let
\be\label{azw1X23}  I_1 (f)=\intl_{|x|>a} \!\frac{|f(x)|}{|x|}\, dx<\infty\quad \mbox{for  all $a>0$}.\ee

\noindent (i) \  Suppose that  $f_{m,\mu} (r)=0$ for almost all $r>0$ if $m=0,1$, and
\be \label{azw2INTRO}
f_{m,\mu} (r)= \sum_{\substack{k=0 \\  m-k \,  even }}^{m-2} \frac{c_{k}}{r^{n+k}}, \qquad c_{k}=\const,\ee
if $m\ge 2$. Then  $(Rf)(\xi)=0$ for almost all hyperplanes $\xi$ in $\rn$.

\noindent (ii) \ Conversely, let $(Rf)(\xi)=0$ for almost all hyperplanes $\xi$ in $\rn$. Suppose, in addition to (\ref{azw1X23}), that
\be \label{azw2INTRO2} I_2(f)\!=\!\intl_{|x|<a}\! \!\!|x|^{N-1} |f(x)|\, dx\!<\!\infty \quad \mbox{for some $N\!>\!0$ and  $a\!>\!0$}.\ee
  Then each  Fourier-Laplace coefficient $f_{m,\mu} (r)$ is a finite linear combination of functions $r^{-n-k}$, $k=0,1, \ldots$, and the following statements hold.

\noindent {\rm (a)} If $m=0,1$, then $f_{m,\mu} (r)\equiv 0$.

\noindent {\rm (b)}  If $m\ge 2$ and $f \neq 0$\footnote{The inequality $f \neq 0$ means  that the set $\{x\in \rn: f (x)\neq 0\}$ has positive measure.},  then $f_{m,\mu} (r)\not\equiv 0$  for at least one pair $(m,\mu)$. For every such pair, $f_{m,\mu} (r)$ has the form (\ref{azw2INTRO}).
\end{theorem}

 The  results of Section 4  trace back to the  aforementioned  works by Helgason  \cite {H63, H64, H65,H11}
and Ludwig \cite{Lud}. Regretfully, some  important justifications in \cite{Lud} are skipped. For example, it is not explained why the functions $\psi_{j,\ell} (s)$ in the proof of Theorem 4.2 (see \cite[p. 65]{Lud}) do exist.  We circumvent this difficulty and suggest a  different approach which invokes  the  Semyanistyi-Lizorkin spaces of Schwartz functions orthogonal to all polynomials.

The condition (\ref{azw2INTRO2}) allows $f(x)$  to  grow as $x\to 0$, but not faster than some power of $|x|^{-1}$.  We conjecture that this condition can be omitted; see open problems at the end of the paper. The finiteness  of $I_1$ is necessary for the existence of the Radon transform  on  radial functions.

The question about the kernel of the Radon transform $f \to Rf$ is closely related to  non-injectivity of this operator when f does not belong to  $L^p(\rn)$ with $1\le p<n/(n-1)$ (otherwise, $R$ is injective). Our Theorem \ref{azw1a223} implies  counter-examples in Boman \cite[Section 6]{Bom91} and Boman and  Lindskog \cite[Section 5]{Bom09}. However, the conclusion (ii) of Theorem \ref{azw1a223} may fail  in the absence of the assumption (\ref{azw1X23}). Indeed, according to  Zalcman \cite{Za82} ($n=2$), Armitage \cite{Arm},
   Armitage and Goldstein \cite{ArG} (see also  Helgason \cite[p. 19]{H11}), there exists a nonconstant harmonic function $h$ on $\rn$, $n\ge 2$, such that $\int_\xi |h|<\infty$ and $\int_\xi  h=0$ for every $(n-1)$-dimensional hyperplane $\xi $. Such a function $h$ obviously satisfies (\ref{azw2INTRO2}). However, the Fourier-Laplace coefficients $h_{m,\mu} (r)$, $m\ge 2$, cannot have the form (\ref{azw2INTRO}). Indeed,   for any $a>0$,
\bea &&\left |\intl_0^a h_{m,\mu} (r)\, dr\right |\le\intl_0^a dr \intl_{\sn} |h(r\th) Y_{m,\mu} (\th)|\, d\th\nonumber\\
&&=\intl_{|x|<a} |h(x) Y_{m,\mu} (x/|x|)|\, \frac{dx}{|x|^{n-1}}\le c\intl_{|x|<a} \frac{dx}{|x|^{n-1}}<\infty.\nonumber\eea
On the other hand, the integral $\int_0^a h_{m,\mu} (r)\, dr$ diverges for all non-zero functions of the form (\ref{azw2INTRO}) because of the strong singularity at $r=0$. This contradiction shows that $h$ does not obey (\ref{azw1X23}).

 Analogues of Theorems \ref{Helgason's2} and \ref{azw1a223} hold for the dual Radon transform; see Theorems \ref{zaeh}, \ref{zaehle4}, \ref{azw1a2R2}. Diverse modifications and generalizations of the Helgason  support theorem can be found
in  \cite{ Bom91a}-\cite{BQ}, \cite{GQ94, OTak}, \cite{Q92}-\cite{Q08}, \cite{Taki1}-\cite{Wie}. The uniqueness problem for  Radon-like  transform was studied in  \cite{Bom91, Bom93, BQ87, Per, Q83}. These publications contain many other related references.
 The methods,  aims, and  results of these works essentially differ from ours.

In Sections 5-8 we study similar problems  for   some other important Radon-like transforms.
Section 5 is devoted to the spherical mean transform that assigns to a function $f$ on $\rn$ the integrals of $f$ over  spheres passing through the origin.   In the case $n=2$, this transform was introduced by  Cormack \cite{Co63} who obtained a formal inversion formula in terms of the Fourier series. A similar inversion problem for spheres through the origin in $\rn$,   $n\ge 3$ odd,   was  studied by  Chen \cite{Che64, Che67}  and Rhee \cite{Rh70, Rh70a}  in connection with  the Darboux equation.
 Their consideration relies on certain paraboloidal means. The case of all $n\ge 2$ was investigated
  by Cormack and Quinto \cite{CoQ}. These authors used spherical harmonic expansions, the link with the dual  Radon transform in $\rn$, and the results of Ludwig \cite{Lud}; see also Quinto \cite{Q82, Q83, Q08} and Solmon \cite[p. 340]{So87}. Our treatment of this class of operators (see Theorems \ref{zaehQuin}, \ref{zaehQusu})
  also relies on the connection with the Radon transform but we do not use the results from \cite{Lud} and deal with more general classes of   functions.

  A similar work has been done in Sections 6,7, and 8 for the Funk transform on the unit sphere $S^n$ \cite {Fu11, Fu13, GGG2, H11},   the corresponding  spherical slice transform for geodesic spheres through the north pole, and  the  totally geodesic Radon transform on the $n$-dimensional real hyperbolic space \cite {H11}.

  The name {\it spherical slice transform} was adopted by  Helgason for the  transformation previously studied by Abouelaz and  Daher \cite{AbD} on zonal functions. In the case of the $2$-sphere in $\bbr^3$ it was proved (see Helgason \cite[p. 145]{H11})  that there is a link  between the spherical slice transform and the Radon transform over lines in the $2$-plane.
   This fact is generalized in Theorems \ref{zasliep} and \ref{zasli} for the $n$-dimensional case, $n\ge 2$,  and combined  with the corresponding statements from Section 4.

   We conclude this discussion by noting that the idea of the projective equivalence of  Radon-like transforms, as in Sections 5-8, is not new;
   cf. \cite{GGG2, P1}. The link between the  Radon transform on $\rn$ and the corresponding transforms on other constant curvature spaces
 was used by Kurusa  \cite{Ku94} to transfer Helgason's support theorem from $\rn$ to the sphere and the hyperbolic space; see also \cite{BC93, BCK}.
  Our formulas are  different and the functions may not be smooth. Moreover, we describe the kernel of the corresponding operators and give examples of their non-injectivity.

   An open problem related to the unilateral structure of the Gegenbauer-Chebyshev integrals and the corresponding Radon transforms is formulated at the end of the paper; see also   an open problem  of the same nature at the end of Section 3.

\section{Preliminaries}

\subsection {Notation}  In the following $\bbz, \bbn, \bbr, \bbc$ are the sets of all integers,
positive integers, real numbers, and complex numbers,
respectively; $\bbz_+ = \{ j \in \bbz: \ j \ge 0 \}$;  $ \bbr_+ = \{ a \in \bbr: \ a > 0 \}$;
$S^{n-1} = \{ x \in \bbr^n: \ |x| =1 \}$ is the unit sphere
in $\bbr^n= \bbr e_1 \oplus \cdots \oplus \bbr e_{n}$, where $e_1, \ldots,  e_{n}$ are the coordinate unit vectors.   For $\theta \in S^{n-1}$, $d\theta$ denotes  the surface element on $S^{n-1}$;  $\sigma_{n-1} =  2\pi^{n/2} \big/ \Gamma (n/2)$ is the surface area of $S^{n-1}$.  We set $d_*\theta= d\theta/\sigma_{n-1}$ for  the normalized surface element on $S^{n-1}$.

The letter $c$  denotes  an inessential positive constant that may vary at each  occurrence. Dealing with  integrals, we say that  the integral
 exists in the Lebesgue sense if it is finite when the expression under the sign of integration is replaced by its absolute value.

\subsection {  Gegenbauer and Chebyshev polynomials}   The Gegenbauer polynomials
$C_m^\lam (t)$ form an orthogonal system in the weighted space $L^2 ([-1,1]; w_{\lam})$, $w_{\lam}(t)=(1-t^2)^{\lam -1/2}$, $\lam >-1/2$.
  In the case  $\lam=0$, they are usually substituted by the Chebyshev polynomials $T_m (t)$. For further references, we review some  properties of the polynomials $C_m^\lam (t)$ and  $T_m (t)$.

Let $|t|\le 1$ and $\lam > -1/2$. Then
\be\label{kioxsru}  |C^\lam_m (t)|\le c\, \left\{\begin{array}{ll} 1,  & \mbox{if $m$ is even,}\\
|t|,  & \mbox{if $m$ is odd,} \qquad c\equiv c(\lam, m) =\const.\\
 \end{array}
\right.\ee
The same inequality holds for $T_m (t)$; cf. 10.9(18) and 10.11(22) in \cite{Er}.
The following equalities  for the Mellin transforms are simple consequences of
  47(1) and 48(4) from  \cite[Sec. 10 (10)]{Mari}. Let $\eta=0$ if $m$ is even and $\eta=1$ if $m$ is odd,
  \be\label {gynko}
c_{\lam,m}=\frac{ \Gam (2\lam+m)\,\Gam (\lam+1/2)}{2 m!\, \Gam (2\lam)}, \qquad \lam > -1/2, \quad \lam \neq 0.\ee
Then\footnote{If $m$ is odd, then $\eta =1$  and  (\ref {89zse55})  is understood  for $-1<  Re\, z \le 0$ by continuity (the  same for (\ref {89zset})).}
\bea\label {89zse} \a_m(z)&\equiv&  \intl_0^1 u^{z-1} (1 \!- \! u^2)^{\lam -1/2}\, C^\lam_m (u)\, du\\
\label {89zse55}&=&
\frac {c_{\lam,m}\, \displaystyle{\Gam \left(\frac{z}{2}\right)\,\Gam \left(\frac{z \!+ \!1}{2}\right)}}
{\displaystyle{ \Gam \left(\lam \! + \!\frac{z\!+ \!1 \!+m}{2}\right)\, \Gam \left(\frac{z \!+ \!1- \!m}{2}\right)}}, \quad Re\, z\!>\!-\eta;\eea
\bea\label {89zse1} \b_m(z)&\equiv& \intl_0^1 u^{z-1} (1 \!- \!u^2)^{\lam -1/2}\, C^\lam_m (1/u)\, du\\&=&
\frac {c_{\lam,m}\, \displaystyle{\Gam \left(\frac{z\!-\!m}{2}\right)\,\Gam \left(\lam \!+\!\frac{z\!+\!m}{2}\right)}}{ \displaystyle{\Gam \left(\lam\!+\!\frac{z}{2}\right)\, \Gam \left(\lam\!+\!\frac{z\!+\!1}{2}\right)}}, \qquad Re\, z>m. \nonumber\eea
These formulas can be equivalently written  in a different form; see 2.21.2(5) and 2.21.2(25) in \cite{PBM2}.
Similarly, by   18(1) and 19(4) from  \cite[Sec. 10 (10)]{Mari}, we have
\be\label {89zset}  \intl_0^1 u^{z-1} (1 \!- \! u^2)^{-1/2}\, T_m (u)\, du=
\frac { \displaystyle{\pi^{1/2}\, \Gam \left(\frac{z}{2}\right)\,\Gam \left(\frac{z \!+ \!1}{2}\right)}}
{\displaystyle{2\, \Gam \left(\frac{z\!+ \!1 \!+m}{2}\right)\, \Gam \left(\frac{z \!+ \!1- \!m}{2}\right)}}, \ee
\be\label {89zse1t}  \intl_0^1 u^{z-1} (1 \!- \!u^2)^{-1/2}\, T_m (1/u)\, du\!=\!
\frac {\displaystyle{\pi^{1/2}\, \Gam \left(\frac{z\!-\!m}{2}\right)\,\Gam \left(\frac{z\!+\!m}{2}\right)}}{ \displaystyle{2\,\Gam \left(\frac{z}{2}\right)\, \Gam \left(\frac{z\!+\!1}{2}\right)}}, \ee
where $Re\, z>-\eta$ and $Re\, z>m$, respectively.

\subsection {Riemann-Liouville and Erd\'elyi-Kober Fractional Integrals}\label{Fractional}
 We briefly review some facts from \cite{Ru13b, SKM}.
For a function $f$  on  $\bbr_+$,  the  fractional integrals  of the Riemann-Liouville type are defined by
\[ (I^\a_{+}f ) (t) =
\frac{1}{\Gamma (\alpha)} \intl^t_0 \frac{f(s) \,ds} {(t - s)^{1-
\alpha}},  \quad (I^\a_{-}f ) (t) = \frac{1}{\Gamma (\alpha)} \intl_t^{\infty}
\frac{f(s) \,ds} {(s-t)^{1- \alpha}}, \]
where $t>0$ and $\a>0$. Both integrals are unilateral.
Hence, the behavior of $f(s)$ is irrelevant for $s\to \infty$ (in $I^\a_{+}f$) and $s\to 0$ (in $I^\a_{-}f$).
   Passing to reciprocals, one can express one integral through another:
 \be\label {zhuhjtf} (I^\a_{-}f)(x)=x^{\a-1} (I^\a_{+}f_1)(1/x), \qquad f_1(x)=x^{-\a-1}f(1/x).\ee
The integral $I^\a_{+}f$ is well defined for any locally integrable
function $f$. The convergence of   $I^\a_{-}f$   depends on    the behavior of $f$  at infinity.
 \begin{lemma}\label{lif} Let $a>0$. If
 \be\label{for10} \intl_a^\infty |f(s)|\, s^{\a -1}\, ds
<\infty,\ee
then  $(I^\a_{-}f)(t)$ is finite for almost all $t> a$.
If $f$ is non-negative, locally integrable on $[a,\infty)$, and (\ref{for10}) fails, then $(I^\a_{-}f)(t)=\infty$ for every $t\ge a$.
\end{lemma}
\begin{proof}
The first statement  is a consequence of the inequality
\[\intl_a^b (I^\a_{-}|f|)(t)\,dt<\infty \quad \forall \,a<b<\infty. \] The latter can be checked  by
 changing the order of integration. To
prove the second statement, we assume the contrary, that is, $f\ge 0$ is locally integrable on $[a,\infty)$, (\ref{for10}) fails, but
$(I^\a_{-}f)(t)$ is finite for some $t\ge a$. Let first $\a\le 1$. Then for any $N>t$,
\[ \intl_t^{N}\frac{f(s)\,
ds}{(s-t)^{1-\a}}>\intl_t^{N}\frac{f(s)\,
ds}{s^{1-\a}}=\left ( \,\intl_a^{N}-\intl_a^{t}\right
)\frac{f(s)\, ds}{s^{1-\a}}.\]
If $N \to \infty$, then,
by the assumption, the left-hand side remains bounded, whereas the
right-hand side tends to infinity.
  If $\a>1$, we proceed as follows. Fix any $b>t$. Then, for any $N>0$,
\[
 \intl_t^{2b+N}\frac{f(s)\,
ds}{(s-t)^{1-\a}}>\intl_{2b}^{2b+N}\frac{f(s)\,
ds}{(s-t)^{1-\a}}>2^{1-\a}\intl_{2b}^{2b+N}\frac{f(s)\,
ds}{s^{1-\a}}\]
 (note that $s-t>s-b>s/2$). The rest of the proof is
as before.
\end{proof}

The corresponding operators $\Cal D^\a_{\pm}$ of fractional differentiation  are defined as left
inverses of  $I^\a_{\pm}$, so that
$\Cal D^\a_{\pm}I^\a_{\pm} f=f$.  The operators  $\Cal D^\a_{\pm}$ may have different analytic forms. For example, if $\alpha = m +
\alpha_0, \;
 m = [\alpha]$ (the integer part of $\a$),  $0 \le \alpha_0 <
1$, then
 \be\label{frr+}\Cal D^\a_{\pm} \vp = (\pm d/dt)^{m +1} I^{1 -
\alpha_0}_{\pm} \vp. \ee
The equality $\Cal D^\a_{\pm}I^\a_{\pm} f=f$ must be justified at each occurrence.

The Erd\'elyi-Kober type fractional integrals are defined by
\be\label{EKO} \!\!\!\!(I^\a_{+, 2} f)(t)\!=\!
\frac{2}{\Gamma (\alpha)} \intl^t_0 \!\!\frac{f(s) \, s\,ds} {(t^2 \!-\! s^2)^{1-
\alpha}},  \quad \!(I^\a_{-, 2} f)(t)\! =\! \frac{2}{\Gamma (\alpha)} \intl_t^{\infty}\!\!
\frac{f(s) \, s\,ds} {(s^2\!-\!t^2)^{1- \alpha}}, \ee
so that $ I^\a_{\pm, 2} f=A^{-1}I^\a_{\pm} Af$ where $(Af) (t)=f (\sqrt {t})$.
 The integral $(I^\a_{+, 2} f)(t)$ is absolutely convergent for almost all $t>0$ whenever $r\to rf(r)$ is a locally integrable function on $\bbr_+$.  For $(I^\a_{-, 2} f)(t)$, the following statement  is a consequence of Lemma \ref{lif}.

 \begin{lemma}\label{lifa2}  Let $a>0$. If
 \be\label{for10z} \intl_a^\infty |f(s)|\, s^{2\a -1}\, ds
<\infty, \ee
 then $(I^\a_{-, 2} f)(t)$ is finite for almost all $t>a$.
  If  $f$ is non-negative, locally integrable on $[a,\infty)$, and (\ref{for10z}) fails, then $(I^\a_{-, 2} f)(t)=\infty$ for every $t\ge a$.
 \end{lemma}

Fractional derivatives $ \Cal D^\a_{\pm, 2}$  of the  Erd\'elyi-Kober type are defined as the  left inverses of $(I^\a_{\pm, 2})^{-1}$. For example, if $\alpha = m +
\alpha_0, \;
 m = [\alpha], \; 0 \le  \alpha_0 <
1$, then, formally, (\ref{frr+}) yields
\be\label{frr+z}   \Cal D^\a_{\pm, 2} \vp=(\pm D)^{m +1}\,
I^{1 - \alpha_0}_{\pm, 2}\vp, \qquad D=\frac {1}{2t}\,\frac {d}{dt}.\ee
The equality $\Cal D^\a_{\pm, 2}I^\a_{\pm, 2} f=f$ must be justified at each occurrence.

Inversion of  $I^\a_{-, 2}$ may cause  difficulties related to convergence at infinity. The following statement holds.
\begin{theorem}\label{78awqe} \cite{Ru13b} Let $\vp= I^\a_{-, 2} f$, where $f$  satisfies   (\ref{for10z}) for every $a>0$. Then  $f(t)= (\Cal D^\a_{-, 2} \vp)(t)$ for almost all $t\in \bbr_+$, where  $\Cal D^\a_{-, 2} \vp$ has one of the following forms.

\noindent {\rm (i)} If $\a=m$ is an integer, then

\be\label {90bedri}
\Cal D^\a_{-, 2} \vp=(- D)^m \vp, \qquad D=\frac {1}{2t}\,\frac {d}{dt}.\ee

\noindent {\rm (ii)}  If $\alpha = m +\alpha_0, \; m = [ \alpha], \; 0 \le \alpha_0 <1$, then

\be\label{frr+z3} \Cal D^\a_{-, 2} \vp = t^{2(1-\a+m)}
(- D)^{m +1} t^{2\a}\psi, \quad \psi=I^{1-\a+m}_{-,2} \,t^{-2m-2}\,\vp.\ee
Alternatively,
\be\label{frr+z4} \Cal D^\a_{-, 2} \vp=2^{-2\a}\, \Cal D^{2\a}_- \, t\,  I^\a_{-, 2}\, t^{-2\a-1} \, \vp,\ee
where $\Cal D^{2\a}_-$ denotes the Riemann-Liouville derivative of order $2\a$, which can be computed according to (\ref{frr+}).

\noindent {\rm (iii)} If, moreover, $\int_1^\infty |f(t)|\, t^{2m +1}\, dt
<\infty$, then
\be\label{frr+z4y0} \Cal D^\a_{-, 2} \vp=(- D)^{m +1} I^{1-\a+m}_{-, 2}\, \vp.\ee
\end{theorem}

The powers of $t$ in this theorem denote the corresponding multiplication operators. An advantage of the inversion formula  (\ref{frr+z4}) in comparison with (\ref{90bedri}),
 (\ref {frr+z3}), and (\ref {frr+z4y0}), is that
 it employs the derivative $d/dt$, rather than  $D=(2t)^{-1} d/dt=d/dt^2$.

\subsection{A Simple Lemma}

The following  lemma, which connects the integration over $S^{n-1}\subset \rn$ with the integration over the coordinate hyperplane $\bbr^{n-1}=\bbr e_1 \oplus \cdots \oplus \bbr e_{n-1}$, is useful in different occurrences.

\begin{lemma} \label {viat} ${}$ \

\noindent ${\rm {(i)}}$ If $f\in L^1(S^{n-1})$, then

  \be\label{teq1} \intl_{S^{n-1}}
f (\theta)\, d\th=\intl_{\bbr^{n-1}}\left [f \Big (\frac{x+e_n}{|x+e_n|} \Big )+f \Big
(\frac{x-e_n}{|x-e_n|} \Big )\right ]\,\frac{dx}{(|x|^2 +1)^{n/2}}\,.\ee
 In particular, if $f$ is even, then
\be \label{hvar} \intl_{S^{n-1}} f(\theta)\,
d\th=2\intl_{\bbr^{n-1}} f\left ( \frac{x+e_n}{|x+e_n|}\right )\,\frac{dx}{(|x|^2 +1)^{n/2}}.\ee

\noindent ${\rm {(ii)}}$ Conversely, if $f\in L^1(\bbr^{n-1})$, $S^{n-1}_+\!=\!\{\theta\in S^{n-1}: \, \theta_n\!>\!0\}$, then
 \be\label{teq2} \intl_{\bbr^{n-1}} f(x)\, dx=\intl_{S^{n-1}_+}f
\Big (\frac{\theta'}{\theta_n}\Big )\,
\frac{d\th}{\theta_n^n}, \qquad \theta'=(\theta_1, \dots, \theta_{n-1}).\ee
\end{lemma}
\begin{proof}  The slice integration yields
$$
\intl_{S^{n-1}} f (\theta)\, d\theta=\intl_0^\pi  \sin ^{n-2} \vp\,
d\vp\intl_{S^{n-2}} f (\om\sin\vp +e_n \,\cos\,\vp)\,  d\om.$$ Set $s=\tan \vp$ on the right-hand side to obtain
$$
\intl_0^\infty \frac{s^{n-2}\, ds}{(1+s^2)^{n/2}}\intl_{S^{n-2}}\left [ f
\Big (\frac{s\om +e_n}{\sqrt{1+s^2}} \Big )+f \Big (\frac{s\om
-e_n}{\sqrt{1+s^2}} \Big )\right ]\, d\om.$$ This coincides with
 (\ref{teq1}). Similarly,
\bea
 &&\intl_{\bbr^{n-1}} f(x)\, dx=\intl_0^\infty s^{n-2}\, ds\intl_{S^{n-2}}
f (s\om)\, d\om\nonumber \\
&=&\intl_0^{\pi/2}\frac{\sin ^{n-2} \vp}{\cos^n \vp }\,
d\vp\intl_{S^{n-2}}
 f \Big (\frac{\om\sin\vp}{\cos\, \vp}\Big )\, d\om=\intl_{S^{n-1}_+}f
\Big (\frac{\theta'}{\theta_n}\Big )\,
\frac{d\th}{\theta_n^n}.\nonumber\eea
 \end{proof}

 \subsection{The Radon Transforms} \label{90kiy3} We recall some facts that are needed for our treatment. More information can be found, e.g., in  \cite{GGG2, GGV, H11, Na1, Ru04b, Ru13b}.
Let $\Pi_n$ be the set of all unoriented hyperplanes in
$\bbr^n$.
The Radon transform of a function $f$ on $\rn$
  is defined by the formula
\be\label{rtra} (Rf)(\xi)= \intl_{\xi} f(x) \,d_\xi x,\qquad \xi \in
\Pi_n,\ee
provided that this integral exists. Here $d_\xi x$ denotes  the Euclidean volume element  in $\xi$.
Every hyperplane $\xi\in \Pi_n$ has the form $\xi=\{x:\, x\cdot \th =t\}$, where $\theta \in \sn$, $t\in \bbr$.
Thus, we can write (\ref{rtra}) as
  \be\label{rtra1}  (Rf) (\theta, t)=\intl_{\theta^{\perp}} f(t\theta +
u) \,d_\theta u,\ee where
$\theta^\perp=\{x: \, x \cdot \theta=0\}$ is the hyperplane
orthogonal to $\theta$ and passing through the origin,
$d_\theta u$ is the Euclidean volume element  in $\theta^{\perp}$. We denote $Z_n=S^{n-1}\times \bbr$ and equip $Z_n$ with the product measure
$d_*\theta dt$, where $d_*\theta=\sig_{n-1}^{-1} d\theta$ is the normalized surface measure on $\sn$.

Clearly, $ (Rf) (\theta,t) = (Rf)(-\theta, -t)$ for every $(\theta,t) \in Z_n$.
 Using (\ref{teq2})  and assuming $t\neq 0$, one can also write
 (\ref{rtra1}) as an integral over  the hemisphere:
 \be\label{hpplz3}
 (Rf) (\theta, t)=|t|^{n-1}\intl_{v\in \sn:\,v\cdot \theta>0} f\left (\frac{tv}{v\cdot \theta} \right )\,\frac{dv}{(v\cdot \theta)^n};\ee
see also  \cite[p. 26]{Na1}.  If $f$ is a  radial function, that is,  $f(x) \!\equiv \!f_0(|x|)$, then $(Rf)(\theta, t) \equiv F_0(t)$, where
\be\label{rese}
F_0(t)=\sigma_{n-2} \intl^\infty_{|t|}\! f_0 (r)
(r^2-t^2)^{(n-3)/2}r dr.\ee

The next theorem shows for which functions $f$  the Radon transform $Rf$ does exist (cf. \cite[Theorem 3.2]{Ru13b}).
\begin{theorem}\label{byvs1} If
\be \label {lkmux}
\intl_{|x|>a}  \,\frac{|f(x)|} {|x|} \, dx<\infty \quad \forall \, a>0, \ee
 then  $(Rf)(\xi)$ is finite for  almost all $\xi \in \Pi_n$.
If $f$ is nonnegative, radial, and (\ref{lkmux}) fails, then  $(Rf)(\xi)\equiv \infty$.
\end{theorem}

The following equality is a particular case of  \cite[formula (2.19)]{Ru04b}:
\be\label{duas3} \intl_{Z_n} \frac{(Rf)(\theta,
t)}{(1+t^2)^{n/2}}\,d_*\theta dt= \intl_{\bbr^n}
\frac{f(x)}{(1+|x|^2)^{1/2}}\,dx \ee
provided that the right-hand side  exists in the Lebesgue sense.

The    dual Radon transform  takes a  function $\vp (\theta,t)$ on $Z_n$  to a function  $(R^*\vp)(x)$ on  $\bbr^n$  by the formula
\be\label{durt}
(R^*\vp)(x)= \intl_{S^{n-1}} \vp(\theta, x\cdot \theta)\,d_*\theta. \ee
The operators $R$ and $R^*$ can be expressed one through another.

\begin{lemma}\label {iozesf} Let $x\neq 0$, $t\neq 0$,
\be \label {jikbVF} (A\vp)(x)\!=\! \frac{1}{|x|^n} \,\vp\left (\frac{x}{|x|}, \frac{1}{|x|}\right ), \qquad (Bf)(\theta, t)\!=\!\frac{1}{|t|^n} \, f \left (\frac{\theta}{t}\right ).\ee
The following equalities hold provided that the expressions on either side exist in the Lebesgue sense:
\be \label {jikb}(R^* \vp)(x)\!=\!\frac{2}{|x|\, \sig_{n-1}}\, (RA\vp)\left (\frac{x}{|x|}, \frac{1}{|x|}\right ), \ee
\be\label {jikbBU} (Rf)(\theta, t)\!=\!\frac{\sig_{n-1}}{2|t|}\, (R^* Bf)\left (\frac{\theta}{t}\right ).\ee
\end{lemma}
\begin{proof} The proof relies on Lemma \ref{viat}.    By (\ref{rtra1}),
\bea
 (RA\vp) (\theta, t)&=&\intl_{\theta^{\perp}} \vp \left(\frac{t\theta +u}{|t\theta +u|}, \frac{1}{|t\theta +u|}\right) \,\frac{d_\theta u}{|t\theta +u|^n}\nonumber \\
&=&\intl_{\bbr^{n-1}} \vp \left(\frac{\gam (e_n  +y)}{|e_n  +y|}, \frac{1}{t|e_n  +y|}\right) \,\frac{dy}{t|e_n  +y|^n},\nonumber \eea
where $\theta =\gam e_n$, $\gam \in O(n)$. Setting $\theta =x/|x|$, $t=1/|x|$, we note that
$$
 \frac{1}{t|e_n  +y|}=|x|\left (e_n \cdot \frac{e_n  +y}{|e_n  +y|}\right )=|x|\left (\gam e_n \cdot \frac{\gam (e_n  +y)}{|e_n  +y|}\right )=
x\cdot \frac{\gam (e_n  +y)}{|e_n  +y|}.$$
Hence,
$$
 (RA\vp)\left (\frac{x}{|x|}, \frac{1}{|x|}\right ) =\intl_{\bbr^{n-1}}
 \vp \left(\frac{\gam (e_n  +y)}{|e_n  +y|}, x\cdot \frac{\gam (e_n  +y)}{|e_n  +y|}\right) \,\frac{|x|\,dy}{|e_n  +y|^n}. $$
 Now (\ref{hvar}) yields
$$
 (RA\vp)\left (\frac{x}{|x|}, \frac{1}{|x|}\right ) =\frac{|x|\, \sig_{n-1}}{2}\,\intl_{S^{n-1}}\vp (\gam\theta, x\cdot \gam\theta)\, d_*\theta=
 \frac{|x|\, \sig_{n-1}}{2}\,(R^* \vp)(x),$$
 which gives (\ref{jikb}). The second equality can be obtained from the first one if we change the notation and assume $\vp$ in  (\ref{jikb}) to be even. Here the cases $t>0$ and $t<0$ should be considered separately.
\end{proof}

Theorem \ref{byvs1} combined with (\ref{jikb}) gives the following
\begin{corollary} \label{gttgzuuuuh} If $\vp (\th,t)$ is  locally integrable on $Z_n$, then the dual Radon transform $(R^*\vp)(x)$ is finite for  almost all $x\in\rn$. If $\vp (\th,t)$ is nonnegative, independent of $\th$, i.e.,  $\vp (\th,t)\equiv \vp_0 (t)$, and such that
\[\intl_0^a \vp_0 (t)\, dt =\infty,\]
for some $a>0$, then  $(R^*\vp)(x)\equiv \infty$.
 \end{corollary}

The following function spaces are  important in the theory of Radon transforms.  Let $S(\bbr^n)$ be the Schwartz space  of $C^\infty$-functions which together with their derivatives of all orders are rapidly decreasing. We supply $S(\bbr^n)$ with the standard topology and denote by $S'(\bbr^n)$ the corresponding space of tempered distributions. The following
 spaces were introduced by  Semyanistyi \cite{Se1}  and extensively studied by
Lizorkin \cite{Liz1}-\cite{Liz3}; see also Helgason \cite{H11} and
Samko \cite{Sa5}. Let
$$
\Psi (\bbr^n)=\{\psi\in S(\bbr^n):\ (\part^j\psi)(0)=0\ \ \text{\rm for all}\quad
j \in \bbz^n_+\}.
$$
We denote by $\Phi(\bbr^n)$  the Fourier image of $\Psi(\bbr^n)$ and supply  $\Psi(\bbr^n)$  and $\Phi(\bbr^n)$ with the topology of  $S(\bbr^n)$. The corresponding spaces of distributions are denoted by $\Psi'(\rn)$ and $\Phi' (\rn)$.
\begin{proposition} \label{lpose6}  Two $S'$-distributions that coincide in the
$\Phi'$-sense differ from each other by a polynomial.\end{proposition}

The analogues  of the Semyanistyi-Lizorkin
 spaces for $Z_n=S^{n-1}\times \bbr$ are defined as follows.   The
derivatives of a function $g$ on $S^{n-1}$ will be defined as the
restrictions onto $S^{n-1}$ of the corresponding derivatives of
$\tilde g(x)= g(x/|x|)$, namely,
\be\label{iiuuyg}
(\partial^\a g)(\theta)=(\partial^\a \tilde g)(x)\big|_
{x=\theta},\qquad \a \in \bbz_+^n, \quad \theta\in S^{n-1}.
\ee
  We denote by $ S(Z_n)$ the space of all functions $\vp(\theta,t)$ on $Z_n=S^{n-1}\times \bbr$, which are
infinitely differentiable in $\theta$ and $t$ and rapidly decreasing
as $t \to \pm\infty$ together with all derivatives.
The topology in $ S(Z_n)$ is defined by the sequence of norms
\be\label {jjyyz}
||\vp ||_m=\sup\limits_{|\a|+j\le m} \sup\limits_{\theta,t} \
\!(1+|t|)^m |(\partial_\theta^\a \partial^j_t \vp)(\theta,t)|, \quad
m\in \bbz_+.
\ee
 The corresponding space of distributions is denoted by $ S'(Z_n)$.
 We set
 \bea \Psi (Z_n)&=&\{\psi(\theta,t)\in
S(Z_n): \ (\partial^\a_\theta
\partial^j_t\psi)(\theta,0)=0, \nonumber\\&{}& \text{for all}\ \a\in\bbz^n_+, \
j\in\bbz_+,
\ \theta\in  S^{n-1}\},\nonumber \\
\label{mest1}\Phi (Z_n)&=&F_1\Psi (Z_n)=\{\vp(\theta,t)\in S(Z_n):
\eea
\[ \intl_{-\infty}^{\infty} t^j (\partial^\a_\theta \partial^k_t \vp)(\theta,t)\,dt=0, \
\text{for all}\  j\in\bbz_+,\ \a\in\bbz^n_+, \ k\in\bbz_+, \
\theta\in S^{n-1}\}.\] Here  $F_1$ denotes the
one-dimensional Fourier transform in the $t$-variable. We supply $\Psi (Z_n)$ and $\Phi (Z_n)$ with the topology of the ambient space $S(Z_n)$. The corresponding spaces of distributions  are denoted by  $\Psi'(Z_n)$ and $\Phi' (Z_n)$. The notation $S_e(Z_n)$, $\Psi_e(Z_n)$, and $\Phi_e(Z_n)$ is used for the corresponding spaces of even functions.

\begin{theorem}\label{pesen}\cite{Se1, H11}  The operator $R $ is an isomorphism from $\Phi(\rn)$ onto $\Phi_e(Z_n)$.  The operator $R^*$ is an isomorphism from  $\Phi_e(Z_n)$ onto $\Phi(\rn)$.
\end{theorem}

\section  {Gegenbauer-Chebyshev Fractional Integrals}\label {1123}

\subsection  {The Right-sided Integrals}

In this section we consider the following integral operators on $\bbr_+$ indexed by $\lam >-1/2$ and a  nonnegative integer $m$.
 Let first  $\lam \neq 0$. We set
\be \label {4gt6a}
(\G^{\lam, m}_{-} f)(t)=\frac{1}{c_{\lam,m}}\,\intl_t^\infty (r^2 - t^2)^{\lam -1/2}\, C^\lam_m \left (\frac{t}{r} \right )\, f (r)  \, r\, dr,\ee
\be \label {4gt6a1}  (\stackrel{*}{\G}\!{}_{\!-}^{\!\lam,m} f)(t)
=\frac{t}{c_{\lam,m}}\,\intl_t^\infty (r^2 - t^2)^{\lam -1/2}\, C^\lam_m \left (\frac{r}{t} \right )\, f (r)  \,\frac{dr}{r^{2\lam +1}},\ee
\be\label {gynko}
c_{\lam,m}=\frac{ \Gam (2\lam+m)\,\Gam (\lam+1/2)}{2 m!\, \Gam (2\lam)}.\ee
In the cases $m=0$ and $m=1$, when  $C^\lam_0 (t) =1$ and $C^\lam_1 (t) = 2\lam t$, these operators are expressed through the Erd\'elyi-Kober type integrals (\ref{EKO}) by the formulas
\be \label {4gt6a8}
\G^{\lam, 0}_{-} f=I_{-,2}^{\lam +1/2}f, \qquad \G^{\lam, 1}_{-} f=t\,I_{-,2}^{\lam +1/2}\,t^{-1} f,\ee
\be \label {4gt6a9}
\stackrel{*}{\G}\!{}_{\!-}^{\!\lam,0} f=t \,I_{-,2}^{\lam +1/2}\,t^{-2\lam -2} f,\quad
\stackrel{*}{\G}\!{}_{\!-}^{\!\lam,1} f=I_{-,2}^{\lam +1/2}\,t^{-2\lam -1} f.\ee
Here, as usual,  the powers of $t$ denote the corresponding multiplication operators.

 In the case  $\lam=0$, when the Gegenbauer polynomials are  substituted by the Chebyshev ones, we set
\be \label {4gt6at}
(\T_{-}^m f)(t)=\frac{2}{\sqrt{\pi}}\,\intl_t^\infty (r^2 - t^2)^{-1/2}\, T_m \left (\frac{t}{r} \right )\, f (r)  \, r\, dr,\ee
\be \label {4gt6a1t}  (\stackrel{*}{\T}\!{}_{\!-}^m f)(t)
=\frac{2t}{\sqrt{\pi}}\,\intl_t^\infty (r^2 - t^2)^{-1/2}\, T_m \left (\frac{r}{t} \right )\, f (r)  \,\frac{dr}{r}.\ee
As in (\ref {4gt6a8}) and (\ref {4gt6a9}).
 \be \label {4gt6a8Ch}
\T^{0}_{-} f=I_{-,2}^{1/2}f, \qquad \T^{1}_{-} f=t\,I_{-,2}^{1/2}\,t^{-1} f,\ee
\be \label {4gt6a9Ch}
\stackrel{*}{\T}\!{}_{\!-}^{\!0} f=t \,I_{-,2}^{1/2}\,t^{-2} f,\quad
\stackrel{*}{\T}\!{}_{\!-}^{\!1} f=I_{-,2}^{1/2}\,t^{-1} f.\ee

  We call  (\ref{4gt6a})-(\ref{4gt6a1})  and (\ref{4gt6at})-(\ref{4gt6a1t})
  the {\it  right-sided Gegenbauer and Chebyshev fractional integrals}, respectively. The next proposition contains information about the existence of these integrals.

   \begin{proposition} \label{lo8xw}  Let $a>0$, $\lam >-1/2$. The integrals $(\G^{\lam, m}_{-} f)(t)$ and $(\stackrel{*}{\G}\!{}_{\!-}^{\!\lam,m} f)(t)$ are finite for almost all $t>a$ under the following conditions.

{\rm (i)} For $(\G^{\lam, m}_{-} f)(t)$:
\be\label{for10zg} \intl_a^\infty |f(t)|\, t^{2\lam-\eta }\, dt
<\infty,\qquad \eta=\left \{\begin{array} {ll} 0 &\mbox{if $m$ is even,}\\
1 &\mbox{if $m$ is odd.}\end{array}\right.\ee

{\rm (ii)} For $(\stackrel{*}{\G}\!{}_{\!-}^{\!\lam,m} f)(t)$:
\be\label{for10zgd} \intl_a^\infty |f(t)|\, t^{m-2}\, dt
<\infty.\ee

 The case $\lam =0$ gives the similar statements for $\T_{-}^m f$ and $\stackrel{*}{\T}\!{}_{\!-}^m f$.
 \end{proposition}
  \begin{proof} (i) follows immediately from Lemma \ref{lifa2} and (\ref{kioxsru}). To prove (ii),   changing the order of integration, for any $b\in (a, \infty)$ we have
  \bea
&&\intl_a^b |(\stackrel{*}{\G}\!{}_{\!-}^{\!\lam,m} f)(t)|\,dt\le c \intl_a^b \frac{dt}{t^2} \intl_t^\infty (r^2 - t^2)^{\lam -1/2}\, \left (\frac{r}{t} \right )^m\, \frac{|f (r)|\,dr}{r^{2\lam +1}}\nonumber\\
&&\le c \intl_a^\infty  \frac{|f (r)|\,dr}{r^{2\lam +1-m}}\intl_{a}^{r}(r^2 - t^2)^{\lam -1/2}\,\frac{dt}{t^{m+2}}\nonumber\\
&&\le
 c \intl_a^\infty  \frac{|f (r)|\,dr}{r^{3}}\intl_{a/r}^{1}(1 - s^2)^{\lam -1/2}\,\frac{ds}{s^{m+2}}\nonumber\\
&&= c \intl_a^\infty \!f(r)\, r^{m-2}\, \eta (r)\,dr,\quad \eta (r)\!=\!r^{-m-1}\intl_{a/r}^{1}(1\! -\! s^2)^{\lam -1/2}\,\frac{ds}{s^{m+2}}.\nonumber\eea
Since the function $\eta (r)$ is bounded, the result follows.
\end{proof}
\begin {remark} \label{sharp1} {\rm The conditions (\ref{for10zg}) and (\ref{for10zgd}) are sharp. Suppose, for example, that $m$ is even and let $f_\e(t)=t^{-2\lam -1+\e}$. Then  (\ref{for10zg}) fails if $f=f_\e$ with $\e=0$. The Gegenbauer  integral  $(\G^{\lam, m}_{-} f_\e)(t)$, which can be explicitly evaluated by (\ref{89zse}) if $\e<0$, does not exist for $\e=0$ too. Other cases in Proposition  \ref{lo8xw} can be considered similarly.
}
\end{remark}

Our main concern is the operators $\G^{\lam, m}_{-}$ and  $\T^{m}_{-}$ which play an important role in the study of Radon transforms. Below we discuss  the injectivity of these operators and inversion formulas.

\begin {lemma} \label{89srg} Let $\lam >-1/2$. If $m=0,1$, then $\G^{\lam, m}_{-}$ is injective on $\bbr_+$ in the class of functions satisfying (\ref{for10zg}) for all $a>0$.
 If $m\ge 2$, then $\G^{\lam, m}_{-}$ is non-injective in this class of functions.
Specifically, let  $f_k (t)=t^{-2\lam-k-2}$,  where $k$ is a nonnegative  integer such that $ m-k=2,4,  \ldots$.
 Then $(\G^{\lam, m}_{-} f_k)(t)=0$ for all $t>0$.  The case $\lam =0$ gives the similar statement for $\T_{-}^m f$.
\end {lemma}
\begin{proof} The first statement is obvious from (\ref {4gt6a8}) and (\ref {4gt6a8Ch}) thanks to the injectivity of the Erd\'elyi-Kober operators. In the case $m\ge 2$, changing variables, we  get
\bea
(\G^{\lam, m}_{-}f_k)(t)&=&\frac{t^{-k-1}}{c_{\lam,m}} \intl_0^1 u^{k} (1 \!- \! u^2)^{\lam -1/2}\, C^\lam_m (u)\, du\nonumber\\&=&
t^{-k-1} \frac {\displaystyle{\Gam \left(\frac{k+1}{2}\right)\,\Gam \left(\frac{k+2}{2}\right)}}
{\displaystyle{ \Gam \left(\lam+ 1+\frac{k+m}{2}\right)\, \Gam \left(\frac{k-m+2}{2}\right)}};
\nonumber\eea
cf. (\ref{89zse}). Since the gamma function $\Gam ((k-m+2)/2)$ has a pole when  $k-m+2=0,-2, -4, \ldots,$ the result follows. If  $\lam =0$ the reasoning is similar and relies on (\ref{89zset}).
\end{proof}

Regarding inversion formulas,  if $m=0$ and $1$,  then $\G^{\lam, m}_{-}$ and $\T^{m}_{-}$ are expressed through the Erd\'elyi-Kober
 integrals  (see  (\ref{4gt6a8})) and can be explicitly inverted using Theorem \ref{78awqe} on the class of functions satisfying  (\ref{for10zg}).  We observe that this condition is necessary for the existence of these integrals.

 In the case  $m\ge 2$ some preparation is needed.

\begin{lemma} \label{89n6g} Let $\lam >-1/2$, $m\ge 2$. Suppose that
\be\label{for10zgn} \intl_a^\infty |f(t)|\, t^{2\lam +m-1}\, dt<\infty\quad \forall a>0.\ee
 \noindent {\rm (i)}  If $\lam \neq 0$, then for almost all $t>0$,
\be\label{8vcmk}
(\stackrel{*}{\G}\!{}_{\!-}^{\!\lam,m}  \G^{\lam, m}_{-} f)(t)=2^{2\lam +1}(I_-^{2\lam+1} f)(t).\ee
 \noindent {\rm (ii)} In the case $\lam=0$ we similarly have
\be\label{8vcmkt}
(\stackrel{*}{\T}\!{}_{\!-}^m  \T_{-}^m f)(t)=2(I_-^{1} f)(t).\ee
\end{lemma}
\begin{proof}  (i) To  prove  (\ref{8vcmk}), we change the order of integration on the left-hand side. To justify  application of Fubini's theorem, let us replace all functions on the left-hand side of (\ref{8vcmk}) by their absolute values and make use of Proposition \ref{lo8xw} (ii) together with (\ref{kioxsru}).
 For any $a>0$ and $m$ even, we obtain
 \bea \label{kiowe}I&\equiv& \intl_a^\infty t^{m-2} dt \intl_t^\infty (r^2 - t^2)^{\lam -1/2}\,\Big | C^\lam_m \left (\frac{t}{r} \right )\Big |\, |f (r)|  \, r\, dr\\
&\le& c\intl_a^\infty |f (r)|  \,r\, dr \intl_a^r (r^2 - t^2)^{\lam -1/2}\, t^{m-2} dt\nonumber\\
&=&c\intl_a^\infty |f (r)|  \,r^{2\lam+m-1}\,  \vp_1 (r)\, dr, \quad \vp_1 (r)\!=\!\intl_{a/r}^1 s^{m-2} (1\!-\!s^2)^{\lam -1/2}\,  ds.\nonumber\eea
Since $\vp_1 (r)$ is bounded, then $I<\infty$.  If $m$ is odd, then $| C^\lam_m (t/r)|\le c \,t/r$ in (\ref{kiowe}) and we proceed as above with
\[ \vp_2 (r)=\intl_{a/r}^1 s^{m-1} (1-s^2)^{\lam -1/2}\,  ds.\]
The latter is bounded. For (\ref{8vcmkt}) the argument is similar.

 The above estimates enable us to change
 the order of integration on the left-hand side of (\ref{8vcmk}) and we get
\[ l.h.s.=\frac{1}{c^2_{\lam,m}}\,\intl_t^\infty f(s) I(s,t)\, ds,\]
 \be\label{89n6gz}I(s,t)=st\intl_t^s  (s^2 - r^2)^{\lam -1/2}  (r^2 - t^2)^{\lam -1/2}\,
C^\lam_m \left (\frac{r}{s} \right )\,\, C^\lam_m \left (\frac{r}{t} \right )\,\frac{dr}{r^{2\lam +1}}.\ee
Let us show that
 \be\label{89n6gzc}I(s,t)=\frac{2^{2\lam +1}\,c^2_{\lam,m}}{\Gam (2\lam +1)}\, (s-t)^{2\lam}\ee
where $c_{\lam,m}$ is defined by (\ref{gynko}). Once (\ref{89n6gzc}) is proved, the result follows.

 Setting $\xi=t/s$, we easily get \be I(s,t)=t^{2\lam} I_0 (\xi),\ee
\bea I_0 (\xi)&=& \xi^{1-2\lam}\intl_\xi^1  (1 - u^2)^{\lam -1/2}  \left (1 - \frac{\xi^2}{u^2}\right)^{\lam -1/2}
C^\lam_m (u)\,C^\lam_m \left (\frac{u}{\xi} \right )\,\frac{du}{u^{2}}\nonumber\\
&=& \xi^{1-2\lam} \,(f_1 \diamond f_2)(\xi).\nonumber\eea
Here $f_1 \diamond f_2$ denotes the Mellin convolution of functions
\[f_1 (u)= u^{-1}(1 - u^2)_+^{\lam -1/2}C^\lam_m (u), \qquad f_2 (u)= (1 - u^2)_+^{\lam -1/2}C^\lam_m (1/u).\]
Thus, we have to show that
\be\label {91s56} I_0 (\xi)=\frac{2^{2\lam +1}\,c^2_{\lam,m}}{\Gam (2\lam +1)}\,\left (\frac{1}{\xi} -1\right )_+^{2\lam}.\ee
It suffices to establish the coincidence of the Mellin transform  $\tilde I_0(z)$ with the  Mellin transform of the right-hand side of (\ref{91s56}) for sufficiently large $Re \, z$.
 The formulas (\ref{89zse}) and (\ref{89zse1}) enable us to compute the Mellin transform of  $I_0 (\xi)$.  We have
\bea
\tilde I_0(z)&\equiv& \intl_0^\infty \xi^{z-1}I_0 (\xi)\, d\xi=\tilde f_1 (z+1-2\lam) \,\tilde f_2 (z+1-2\lam)\nonumber\\
&=&c^2_{\lam,m} \frac{\displaystyle{\Gam \left(\frac{z}{2}-\lam\right)\,\Gam \left(\frac{z+1}{2}-\lam\right)}}
{\displaystyle{\Gam \left(\frac{z+1}{2}\right)\, \Gam \left(\frac{z}{2}+1\right)}}=2^{2\lam +1}c^2_{\lam,m}\, \frac{\Gam (z-2\lam)}{\Gam (z+1)}
\nonumber\\
&=&\frac{2^{2\lam +1}\,c^2_{\lam,m}}{\Gam (2\lam +1)}\,\intl_0^\infty \xi^{z-1}\,\left (\frac{1}{\xi} -1\right )_+^{2\lam}\, d\xi.
\nonumber\eea
Thus, the Mellin transforms of the both sides of (\ref{91s56}) coincide and we are done.

(ii) Let us prove (\ref{8vcmkt}).  As above,
\[ l.h.s.=\frac{4}{\pi}\,\intl_t^\infty f(s) I(s,t)\, ds,  \qquad I(s,t)\equiv I_0(\xi)=\xi \,(f_1 \diamond f_2)(\xi), \quad \xi=\frac{t}{s},\]
where
\[f_1 (u)= u^{-1}(1 - u^2)_+^{-1/2}T_m (u), \qquad f_2 (u)= (1 - u^2)_+^{-1/2}T_m (1/u).\]
By   (\ref{89zset}) and (\ref{89zse1t}),
\[
\tilde I_0(z)=\tilde f_1 (z+1) \,\tilde f_2 (z+1)=\frac{\pi}{2z},\]
and therefore,
$$I_0(\xi)=\frac{\pi}{2}\, H (1-\xi)=\frac{\pi}{2}\,\left \{\begin{array} {ll} 1 &  \mbox {if} \quad \xi<1,\\
0 & \mbox {if} \quad \xi>1.\\
\end{array}\right.
$$
 This gives $I(s,t)=(\pi/2)H(1-t/s)$, and (\ref{8vcmkt}) follows.
\end{proof}

The following inversion formulas for the Gegenbauer-Chebyshev integrals are immediate consequences of Lemma \ref{89n6g}.

\begin {corollary}\label {mlpzx} Let $m\ge 2$, $\lam >-1/2$, and suppose that $f$ satisfies the conditions of Lemma \ref {89n6g}.
 Then $f(t)$ can be uniquely reconstructed for almost all $t>0$  from the Gegenbauer-Chebyshev integrals
$\G^{\lam, m}_{-} f=g$ and $ \T_{-}^m f=g$
 by the formulas
\be\label{8vcmks}
f(t)=2^{-2\lam -1}(\Cal D^{2\lam +1}_{-}\stackrel{*}{\G}\!{}_{\!-}^{\!\lam,m}  g)(t), \ee
\be\label{8vcmkst}
f(t)=-\frac{1}{2}\,\frac{d}{dt}\,(\stackrel{*}{\T}\!{}_{\!-}^m \,t^{-2}\, g)(t),\ee
 where $\Cal D^{2\lam +1}_{-}$ stands for the corresponding Riemann-Liouville fractional derivative; see Section \ref{Fractional}.
 \end{corollary}

\begin{remark} \label {yyur} {\rm The assumption $\int_a^\infty |f(t)|\, t^{2\lam +m-1} dt\!<\!\infty$ in Corollary \ref{mlpzx} is essentially stronger than  (\ref{for10zg})
 in Proposition \ref{lo8xw}(i) which guarantees  the existence of $\G^{\lam, m}_{-} f$. The inversion problem for $\G^{\lam, m}_{-} f$ under the  less restrictive assumption (\ref{for10zg}) does not have a unique solution; cf. Lemma \ref{89srg}. We recall that for $m=0$ and $1$, unlike $m\ge 2$, the inversion formulas provided by Theorem \ref{78awqe} hold under the same assumptions which are necessary for the existence of the corresponding Gegenbauer-Chebyshev integrals.}
\end{remark}

\subsection  {The Left-sided Integrals}

 Let $\lam >-1/2$,  $m \in \bbz_+$. The {\it  left-sided Gegenbauer and Chebyshev fractional integrals} are defined as follows.
 For  $\lam \neq 0$, we set
\be \label {4gt6ale}
(\G^{\lam, m}_{+} f)(r)=\frac{r^{-2\lam}}{c_{\lam,m}}\,\intl_0^r (r^2 - t^2)^{\lam -1/2}\, C^\lam_m \left (\frac{t}{r} \right )\, f (t)  \, dt,\ee
\be \label {4gt6a1le}  (\stackrel{*}{\G}\!{}_{\!+}^{\!\lam,m} f)(r)
=\frac{1}{c_{\lam,m}}\,\intl_0^r (r^2 - t^2)^{\lam -1/2}\, C^\lam_m \left (\frac{r}{t} \right )\, f (t) \, t\,dt,\ee
$c_{\lam,m}$ being defined by (\ref{gynko}). In the case $\lam = 0$ we denote
\be \label {4gt6atle}
(\T_{+}^m f)(r)=\frac{2}{\sqrt{\pi}}\,\intl_0^r (r^2 - t^2)^{-1/2}\, T_m \left (\frac{t}{r} \right )\, f (t)  \,  dt,\ee
\be \label {4gt6a1tle}  (\stackrel{*}{\T}\!{}_{\!+}^m f)(r)
=\frac{2}{\sqrt{\pi}}\,\intl_0^r (r^2 - t^2)^{-1/2}\, T_m \left (\frac{r}{t} \right )\, f (t) \, t\,  dt.\ee
The left-sided  integrals are expressed through the right-sided ones by the formulas
\bea &&\label {4gt6a1ldr}
(\G^{\lam, m}_{+} f)(r)=\frac{1}{r}\, (\G^{\lam, m}_{-} f_1)\left(\frac{1}{r}\right), \qquad f_1(t)=\frac{1}{t^{2\lam +2}} \, f\left(\frac{1}{t}\right);\qquad \\
&& \label {4gt6a1ldz}
(\stackrel{*}{\G}\!{}_{\!+}^{\!\lam,m} f)(r)=r^{2\lam}(\stackrel{*}{\G}\!{}_{\!-}^{\!\lam,m} f_2)\left(\frac{1}{r}\right), \qquad f_2(t)=\frac{1}{t} \, f\left(\frac{1}{t}\right).\eea

These formulas combined with Proposition \ref{lo8xw} give the following statement.

 \begin{proposition} \label{lo8xwADD}  Let  $a>0$,  $\lam >-1/2$.  The integrals (\ref{4gt6ale})-(\ref{4gt6a1tle})
 are absolutely convergent  for almost all $r<a$ under the following conditions.

{\rm (i)} For (\ref{4gt6ale}), (\ref{4gt6atle}):
 \be\label {026g094t}   \intl_0^a t^\eta f(t)\, dt <\infty,\qquad \eta=\left \{\begin{array} {ll} 0 &\mbox{if $m$ is even,}\\
1 &\mbox{if $m$ is odd,}\end{array}\right.\ee

{\rm (ii)} For (\ref{4gt6a1le}), (\ref{4gt6a1tle}):
 \be\label {026gt}
  \intl_0^a  t^{1-m} f(t)\, dt <\infty.\ee
 \end{proposition}

The conditions (\ref{026g094t}) and (\ref{026gt}) are sharp; see Remark \ref{sharp1}.
The following statement can be derived from  Lemma \ref{89srg}, using (\ref{4gt6a1ldr}),  or  proved directly, using (\ref{89zse1}).

 \begin {lemma} \label {marti}
If $m=0,1$, then $\G^{\lam, m}_{+}$ is injective on $\bbr_+$  in the class of functions satisfying (\ref{026g094t}) for all $a>0$.
 If $m\ge 2$, then $\G^{\lam, m}_{+}$ is non-injective in this class of functions.
Specifically, let  $f_k (t)=t^k$,  where $k$ is a nonnegative  integer such that $ m-k=2,4,  \ldots$.
 Then $(\G^{\lam, m}_{+} f_k)(t)=0$ for all $t>0$.
\end {lemma}
Similarly, Lemma \ref{89n6g}  yields the following.

\begin{lemma} \label{89n6gle} Let $m\ge 2$, $\lam >-1/2$. Suppose that $f$ satisfies (\ref{026gt}) for all $a>0$.
  If $\lam \neq 0$, then for almost all $t>0$,
\be\label{8vcmkle}
(\stackrel{*}{\G}\!{}_{\!+}^{\!\lam,m}  \G^{\lam, m}_{+} f)(t)=2^{2\lam +1}(I_+^{2\lam+1} f)(t).\ee
In the case $\lam=0$ we similarly have
\be\label{8vcmktle}
(\stackrel{*}{\T}\!{}_{\!+}^m  \T_{+}^m f)(t)=2(I_+^{1} f)(t).\ee
\end{lemma}

\begin {corollary}\label {mlpzxle} Suppose that $\lam >-1/2$ and let $f$ satisfy (\ref{026gt})  for all $a>0$.
 Then $f(t)$ can be uniquely reconstructed for almost all $t>0$  from the Gegenbauer-Chebyshev integrals
 by the formulas
\be\label{8vcmksle}
f(t)=2^{-2\lam -1}  (\Cal D^{2\lam +1}_{+}
\stackrel{*}{\G}\!{}_{\!+}^{\!\lam,m}  g)(t), \qquad g=\G^{\lam, m}_{+} f, \ee
\be\label{8vcmkstle}
f(t)=\frac{1}{2}\,\frac{d}{dt}\,(\stackrel{*}{\T}\!{}_{\!+}^m  g)(t),\qquad g=\T_{+}^m f,\ee
where $\Cal D^{2\lam +1}_{+}$ stands for the corresponding Riemann-Liouville fractional derivative.
\end{corollary}

{\bf Open Problem.} Are there  any other functions in the kernel of  $\G^{\lam, m}_{\pm} $, rather than those  indicated by Lemmas  \ref{marti} and \ref{89srg}? It is assumed that the action of these operators is considered on  functions satisfying the conditions of Propositions \ref{lo8xwADD} and \ref{lo8xw}, respectively.

\section {Radon Transforms and Spherical  Harmonics}\label {ndspher}

We fix a real-valued orthonormal basis $\{Y_{m,\mu}\}$ of spherical harmonics in $L^2(S^{n-1})$; see, e.g., \cite{Mu}. Here
 $m\in \bbz_+$ and $ \mu=1,2, \ldots d_n (m)$, where
\be\label{kWSQRT} d_n(m) =(n+2m-2)\,
\frac{(n+m-3)!}{m! \, (n-2)!}\ee
is the dimension of the subspace of spherical harmonics of degree $m$.

The following Funk-Hecke Theorem is well-known in analysis on the sphere; see,  e.g., \cite[p. 18]{H08},   \cite[p. 117]  {See}.

 \begin{theorem}\label {sTTT1}   Let $h (s) (1-s^2)^{(n-3)/2}\in L^1(-1,1)$. Then  for every spherical harmonic $Y_m$ of degree $m$ and every $\th\in\sn$,
\be \label {fhf}
\intl_{S^{n-1}} h(\theta \cdot \xi)\,Y_m(\xi)\, d\xi=\lam_m \,Y_m (\theta) \qquad \text{(the Funk-Hecke formula)},\ee
where
\be \label {shap232S}
\lam_m=\sig_{n-2}\intl_{-1}^1 h(s)\,P_m(s)\,
(1-s^2)^{(n-3)/2}\, ds,\ee
\be\label {shap30} P_m(s)=
\begin{cases}
T_m (s) \quad &\mbox{if} \quad n=2,\\
{}\\ \displaystyle{\frac{m!\, (n-3)!}{ (m+n-3)!}}\, \,C_m^{n/2 -1}
(s) \quad &\mbox{if} \quad n\ge 3.
\end{cases}
\ee
\end{theorem}

Now let us consider the Radon transform and its dual; see Section \ref{90kiy3}.
Since these transforms commute with rotations, they  can be diagonalized (at least formally) in terms of  spherical harmonic expansions. Specifically, if
\be\label{Xuz}f (x) \sim \sum_{m,\mu} f_{m,\mu} (r)\,Y_{m,\mu} (\theta), \qquad  r=|x|\neq 0, \quad \theta=x/r,\ee then
 for $\vp (\theta,t)=(Rf)(\theta,t)$  we have
\be\label{Xuz1} \vp (\theta,t)\sim \sum_{m,\mu} \vp_{m,\mu} (t)\,Y_{m,\mu} (\theta),\ee
where $\vp_{m,\mu}$ expresses through $f_{m,\mu}$ for each pair $m,\mu$.
Similarly, if $\vp$ is a function on $Z_n=S^{n-1} \times \bbr$,   then, for $f=R^*\vp$, (\ref{Xuz1}) implies (\ref{Xuz}) with $f_{m,\mu}$ determined by $\vp_{m,\mu}$.

We will be using the same notation $\frS_m$ for the spaces of functions of the form $f(x)=u(|x|)\,Y_m (x/|x|)$ and $\vp(\theta,t)=v(t)\,Y_m (\theta)$, where $Y_m$ is a spherical harmonic of degree $m$.

\subsection{Action  on the Spaces $\frS_m$}\label {nic Exp}
${}$
\vskip 0.2truecm

The formulas in this section are not new (at least, for smooth rapidly decreasing functions); cf. \cite{Dea, Lud, Na1}. We present them in our notation and give an independent proof under minimal assumptions related to the existence of the corresponding integrals.

Let  $\lam=(n-2)/2$, $f(x)=u(|x|)\,Y_m (x/|x|)$. For $n\ge 3$ we define
 \be\label {885frqi}
 v(t)=\frac{\pi^{\lam +1/2}}{c_{\lam,m}}\,\intl_{|t|}^\infty (r^2 - t^2)^{\lam -1/2}\, C^\lam_m \left (\frac{t}{r} \right )\, u (r)  \, r\, dr, \ee
 \be\label{gtawe}
c_{\lam,m}=\frac{ \Gam (2\lam+m)\,\Gam (\lam+1/2)}{2 m!\, \Gam (2\lam)}=\frac{(n+m-3)!\, \Gam((n-1)/2)}{2m!\,(n-3)! }.\ee
Similarly, for $n=2$ we set
  \be \label {4gt6atqQ}
v(t)=2\intl_{|t|}^\infty (r^2 - t^2)^{-1/2}\, T_m \left (\frac{t}{r} \right )\, u (r)  \, r\, dr.\ee

\begin{lemma} \label{ppo9j} Let  $f(x)=u(|x|)\,Y_m (x/|x|)$, where
\be\label{oollazs} \intl_a^\infty |u(r)| \,r^{2\lam} dr <\infty\quad \forall a>0.\ee
Then $(Rf)(\theta,t)$ is finite for all $\theta \!\in \!\sn$ and almost all $t\!\in \!\bbr$. Furthermore,
\be\label{12oollazs}(Rf)(\theta,t)=v(t)\, Y_m (\theta).\ee
 The function $v(t)$  has the following properties:

\vskip 0.2truecm

{\rm (a)}   $\quad v(- t)=(-1)^m v(t)$.

\vskip 0.2truecm

{\rm (b)}  If $t>0$, then $v$    is represented by the Gegenbauer-Chebyshev   integrals (\ref{4gt6a}) and (\ref{4gt6at}). Specifically,
\be\label {poxe}
 v(t)=  \pi^{\lam +1/2}  (\G^{\lam, m}_{-} u)(t) \quad \mbox{\rm and } \quad  v(t)= \pi^{1/2}  (\T_{-}^m u)(t)\ee
for  $n\ge 3$  and $n=2$, respectively.

\vskip 0.2truecm

{\rm (c)}  For any nonnegative integer $j<m$,
\be\label {poxe1}
\intl_{-\infty}^\infty t^j v(t)\, dt=0 \quad \mbox{\rm provided that} \quad  \intl_0^\infty |u(r)|\, r^{j+2\lam +1}dr<\infty.\ee
\end{lemma}

\begin{proof} Let first $t>0$. By (\ref{hpplz3}),
\[(Rf) (\theta, t)\!=\!t^{n-1}\!\!\!\intl_{\om\cdot \theta>0} \!\!\!f\left (\frac{t\om}{\om\cdot \theta} \right )\,\frac{d\om}{(\om\cdot \theta)^n}, \quad f\left (\frac{t\om}{\om\cdot \theta} \right)\!=\! u\left (\frac{t}{\om\cdot \theta} \right) Y_m (\om).\]
 Now (\ref{12oollazs})  holds by the Funk-Hecke formula (\ref{fhf}) (set $h(s)=s^{-n} u(t/s)$ if $s>0$ and $h(s)\equiv 0$, otherwise) and
\be\label{78amo} v(t)=\sig_{n-2}\intl_0^1 (1-s^2)^{(n-3)/2}   P_m (s)\, u\left(\frac{t}{s}\right )\,\frac{ds}{s^n}\, .\ee
By (\ref{oollazs})  and Lemma \ref{lifa2},     the condition $h (s) (1-s^2)^{(n-3)/2}\in L^1(-1,1)$ in Theorem \ref{sTTT1} is satisfied  for almost all $t>0$ , so that (\ref{12oollazs})  is valid for all $\theta \!\in \!\sn$ and almost all $t>0$.
 The equality (\ref{78amo}) implies  (\ref {885frqi}) and (\ref{4gt6atqQ}).
The formulas in (\ref{poxe}) follow from (\ref{78amo}) owing to (\ref{4gt6a}) and (\ref{4gt6at}).
  The equality $v(- t)=(-1)^m v(t)$ is a consequence of the formulas $(Rf) (\theta, t)=(Rf) (-\theta, -t)$ and $Y_m (-\theta)=(-1)^m Y_m (\theta)$.
   To prove (c), we first change the order of integration. This operation is possible thanks to  the inequality in (\ref{poxe1}). Then the result follows by the orthogonality of Gegenbauer (or Chebyshev) polynomials.
\end{proof}

For the dual Radon transform we have the following.

\begin{lemma} \label{ppo9q}Let  $\lam=(n-2)/2$, $\vp(\theta,t)=v(t)\,Y_m (\theta)$, where  $Y_m$ is a spherical harmonic of degree $m$ and $v(t)$ is a locally integrable function on $\bbr$ satisfying $v(- t)=(-1)^m v(t)$.  Then $(R^*\vp)(x)\equiv (R^*\vp)(r\theta)$ is finite for all $\theta\in \sn$ and almost all $r>0$. Furthermore,
\be\label {poxe190} (R^*\vp)(r\theta)=u(r)\,Y_m (\th).\ee The function  $u(r)$    is represented by the Gegenbauer  integral (\ref{4gt6ale}) (or the  Chebyshev integral (\ref{4gt6atle})) as follows.

\noindent For $n\ge 3:$
\be\label {885frq} u(r)=\frac{r^{-2\lam}}{\tilde  c_{\lam,m}}\,\intl_0^r (r^2 - t^2)^{\lam -1/2}\, C^\lam_m \left (\frac{t}{r} \right )\, v (t)  \,  dt=   \pi^{\lam +1/2}  (\G^{\lam, m}_{+} v)(t),\ee
\[
\tilde c_{\lam,m}=\frac{\pi^{1/2} \Gam (2\lam+m)\,\Gam (\lam+1/2)}{2 m!\, \Gam (2\lam)\, \Gam (\lam+1)}.\]

\noindent For $n=2:$
\be \label {4gt6atq}
u(r)=\frac{2}{\pi}\intl_0^r (r^2 - t^2)^{-1/2}\, T_m \left (\frac{t}{r} \right )\,  v (t)  \,  dt= \pi^{1/2} (\T_{+}^m v)(t).\ee
\end{lemma}
\begin{proof} We first note that  $\vp$ is locally integrable on $Z_n$ and therefore, $(R^* \vp)(x)$ is
 finite for almost all $x$. Then, by the  Funk-Hecke formula, we get (\ref{poxe190}) with
 \[ u(r)=\frac{\sig_{n-2}}{\sig_{n-1}}\intl_{-1}^1 (1-s^2)^{(n-3)/2} P_m (s)\,v(rs)\,ds. \]
Since $v(- s)=(-1)^m v(s)$ and $P_m (-s)=(-1)^m P_m (s)$, the last formula gives the result.
\end{proof}

\begin{theorem} \label {recon65} Suppose that
 \be\label{oslls} \intl_{|x|>a} |f(x)|\, \left \{\begin{array} {ll} |x|^{-1}   &  \mbox {if} \quad m=0,1,\\
 |x|^{m-2}   &  \mbox {if} \quad m\ge 2\\
 \end{array}\right \}\, dx <\infty \quad \forall a>0.\ee
Then the Fourier-Laplace coefficients $f_{m,\mu}(t)$ of $f$ can be uniquely reconstructed for almost all $t>0$ from  the corresponding coefficients $\vp_{m,\mu}$ of $\vp=Rf$ by the following formulas.

\noindent For $n\ge 3:$
\be\label{012aq}
f_{m,\mu}(t)=c\,  \left (-\frac{d}{dt}\right )^{n-1}t  \intl_t^\infty (r^2 - t^2)^{(n -3)/2}\, C^{n/2 -1}_m \left (\frac{r}{t} \right )\, \vp_{m,\mu} (r)  \, r^{1-n}\, dr,\ee
\[c=\frac{ \Gam (n/2-1)\,m!}{2\pi^{(n -1)/2} (n-3+m)!}.\]

\noindent For $n=2:$
\be\label{012aq2}
f_{m,\mu}(t)=-\frac{1}{\pi}\, \frac{d}{dt}\, t\intl_t^\infty (r^2 - t^2)^{-1/2}\, T_m \left (\frac{r}{t} \right )\, \vp_{m,\mu} (r)  \,
\frac{dr}{r}.\ee
\end{theorem}
\begin{proof} By Lemma \ref{ppo9j},   $\vp_{m,\mu} (t)= \pi^{(n -1)/2}  (\G^{n/2 -1, m}_{-} f_{m,\mu})(t)$ if $n\ge 3$, and
$\vp_{m,\mu} (t)= \pi^{1/2}  (\T_{-}^m f_{m,\mu})(t)$
 if $n=2$.  Hence, the result follows by Corollary \ref{mlpzx}, the conditions of which are satisfied, owing to   (\ref{oslls}).
\end{proof}

\subsection{The Kernel and Support Theorems}\label {pport The}

\subsubsection{The Kernel of $R^*$}

The next two theorems give the description of the kernel of $R^*$  in terms of the  Fourier-Laplace coefficients
\be \label{zaehl2YU}  \vp_{m,\mu} (t)\!=\!\intl_{S^{n-1}}\vp (\th, t)\, Y_{m,\mu} (\th)\, d\th. \ee
In both theorems it is assumed that $\vp (\th, t)$ is an even  locally integrable  function on  $Z_n=S^{n-1} \times \bbr$. The inequality $\vp \neq 0$,  means that  the set $\{(\th,t): \vp (\th, t)\neq 0\}$ has positive measure.

\begin{theorem} \label{zaeh} Let  $\vp_{m,\mu} (t)=0$ for almost all $t\in \bbr$ if $m=0,1$, and
\be\label{KOKLI} \vp_{m,\mu} (t)= \sum_{\substack{k=0 \\  m-k \,  even }}^{m-2} c_k \,t^k, \qquad c_k=\const,\ee
if $m\ge 2$. Then $(R^*\vp)(x)=0$ for almost all $x\in \rn$.
\end{theorem}

\begin{theorem} \label{zaehle4}  Suppose in addition that  $\vp\in S'(Z_n)$. If  $(R^*\vp)(x)=0$ for almost all $x\in \rn$, then all  $\vp_{m,\mu} (t)$ are polynomials and the following statements hold.

\noindent (i) If $m=0,1$, then $\vp_{m,\mu} (t)\equiv 0$.

\noindent (ii)  If $m\ge 2$ and $\vp \neq 0$,  then $\vp_{m,\mu} (t)\not\equiv 0$  for at least one pair $(m,\mu)$. For every such pair, $\vp_{m,\mu} (t)$ has the form (\ref{KOKLI}).
\end{theorem}

 The proof of these theorems   needs some preparation.

\begin{lemma} \label{zaehle} If $\vp \in L^1_{loc}(Z_n)$ is  even, then for almost all $r>0$,
\be \label{zaehl2}
(R^*\vp)_{m,\mu}(r)\!\equiv\!\intl_{S^{n-1}}\!\!(R^*\vp) (r\th)\, Y_{m,\mu} (\th)\, d\th\!= \!\pi^{\lam +1/2}\,(\G^{\lam, m}_{+} \vp_{m,\mu})(r),\ee
 where $\lam=(n-2)/2$ and $\G^{\lam, m}_{+} \vp_{m,\mu}$ is the  Gegenbauer  integral  (\ref{4gt6ale})  (or the  Chebyshev integral (\ref{4gt6atle})).
 \end{lemma}
\begin{proof}
Since the integral in (\ref{zaehl2}) exists in the Lebesgue sense for almost all $r>0$, we can  change the order of integration. Using the Funk-Hecke formula (\ref{fhf}), we obtain
\bea
(R^*\vp)_{m,\mu}(r)&=&\intl_{S^{n-1}} \,d_*\eta\intl_{S^{n-1}} \vp (\eta, r\th \cdot \eta)\,Y_{m,\mu} (\th)\, d\th\nonumber\\
&=&\frac{\sig_{n-2}}{\sig_{n-1}}\intl_{-1}^1 (1-s^2)^{(n-3)/2} P_m (s)\,ds \intl_{S^{n-1}} \vp (\eta, rs)\,Y_{m,\mu} (\eta)\,d\eta \nonumber\\
&=&\frac{\sig_{n-2}}{\sig_{n-1}}\intl_{-1}^1 (1-s^2)^{(n-3)/2} P_m (s)\,\vp_{m,\mu}(rs)\,ds.\nonumber\eea\
Since $\vp$ is even, then
$\vp_{m,\mu}(- t)=(-1)^m \vp_{m,\mu}(t)$. Moreover, $P_m (-s)=(-1)^m P_m (s)$. Hence, the last integral equals $
\pi^{\lam +1/2}\,(\G^{\lam, m}_{+} \vp_{m,\mu})(r)$; cf. the  proof of Lemma \ref{ppo9q}.
\end{proof}

{\it Proof of Theorem \ref{zaeh}} \  By Lemma \ref {marti}, the operator $\G^{\lam, m}_{+}$ annihilates  monomials $t^k$ provided that $ 0\le k\le m-2$ with  $m-k$  even. Hence, by (\ref{zaehl2}),
$(R^*\vp)_{m,\mu}(r)=0$ for almost all $r>0$.  We recall that $R^*\vp$ is locally integrable in $\rn$. Hence, by Fubini's theorem,
 the function $f_r (\th)\equiv (R^*\vp)(r\th)$ belongs to $L^1(S^{n-1})$  for almost all $r>0$. Let us consider the Poisson integral
\[ (\Pi_\rho f_r) (\th)\!=\!\frac{1}{\sig_{n-1}}\intl_{S^{n-1}} \frac{1\!-\!\rho^2}{|\th\!-\!\rho\eta|^n}\,f_r (\eta)\, d\eta;\]
see, e.g., Stein and Weiss \cite{SW}.
Since $(R^*\vp)_{m,\mu}(r)\!=\!0$ a.e. for all $m$, $\mu$, then
$$(\Pi_\rho f_r) (\th)=\sum\limits_{m,\mu}\rho^m  (f_r)_{m,\mu} Y_{m,\mu} (\th)=\sum\limits_{m,\mu}\rho^m  (R^*\vp)_{m,\mu}(r) Y_{m,\mu} (\th)=0$$
for almost all $r>0$, all $\rho \in [0,1)$, and all $\th\in S^{n-1}$. Furthermore,
 since
\[f_r (\th)=\lim\limits_{\rho \to 1}(\Pi_\rho f_r) (\th)\]
 in the $L^1$-norm, then
 $f_r (\th)\!=\!(R^*\vp)(r\th)\!=\!0$ for almost all $\th \in S^{n-1}$ and almost all $r>0$. This gives the result.
\hfill $\square$

Note that in Theorem \ref{zaeh} we did not assume  $\vp \in S'(Z_n)$. This assumption will be used in the proof of Theorem \ref{zaehle4}.

  The next lemma employs the distribution spaces $S'(Z_n)$ and $\Phi'(Z_n)$ from Section \ref{90kiy3}.

 \begin{lemma} \label{zaehle3}  Let $\vp$ be a locally integrable  function in $S'(Z_n)$. If $\vp=0$ in the $\Phi'(Z_n)$-sense, then all  Fourier-Laplace coefficients $\vp_{m,\mu} (t)$ are polynomials. If, moreover, $\vp \neq 0$,  then  $\vp_{m,\mu} (t)\not\equiv 0$  for at least one pair $(m,\mu)$.
\end{lemma}
\begin{proof} Given $\om \in S(\bbr)$, let $\psi (\th, t)= \om (t)Y_{m,\mu} (\th)\in S(Z_n)$. Then the expression
\[(\vp_{m,\mu},\om)=\intl_{Z_n} \vp (\th, t)\,\psi(\th, t)\, d\th dt=(\vp, \psi )\]
is meaningful, that is, $\vp_{m,\mu}\in S'(\bbr)$. If $\om \in \Phi (\bbr)$, then  $\psi\in \Phi(Z_n)$ and, by the assumption,
$(\vp_{m,\mu},\om)=(\vp, \psi )=0$, that is, $\vp_{m,\mu}=0$ in the $\Phi'(\bbr)$-sense. Hence, by Proposition \ref{lpose6}, $\vp_{m,\mu} (t)$ is a polynomial.
 If all $\vp_{m,\mu} (t)$ are identically zero, then  $\vp (\th, t)=0$ a.e. on $Z_n$,  which gives the second statement by contradiction.
\end{proof}

{\it Proof of Theorem \ref{zaehle4}} \   Since $(R^*\vp)(x)=0$, then $(R^*\vp)_{m,\mu}(r)=0$ and, by  (\ref{zaehl2}),
\be\label{zaehle5} (\G^{\lam, m}_{+} \vp_{m,\mu})(r)=0\ee
for almost all $r>0$ and all $m,\mu$. Furthermore,  if $g\in \Phi (Z_n)$ and $g_e\in \Phi_e (Z_n)$ is the even component of $g$, then $(\vp,g)=(\vp,g_e)$, because $\vp$ is even. Since by Theorem \ref{pesen}, $g_e=Rf$ for some $f\in \Phi(\rn)$, then  $R^*\vp=0$ yields
$$
(\vp, g)=(\vp,g_e)=(\vp,Rf)=(R^*\vp, f)=0.
$$
By Lemma  \ref{zaehle3} it follows that $\vp_{m,\mu} (t)$ is a polynomial.
The structure of this polynomial is determined by the equality
\be\label{zaehle6}\G^{\lam, m}_{+} \vp_{m,\mu}=0,\ee
 which follows from (\ref{zaehle5}). Specifically, by Lemma \ref{marti}, if $m=0,1$, then $\vp_{m,\mu}(t)=0$ for  all $t\in \bbr_+$. If, moreover, $\vp \neq 0$, then $\vp_{m,\mu} (t)$ is not identically zero for at least one pair $(m,\mu)$ with $m\ge 2$. For each such  pair,
  $\G^{\lam, m}_{+} \vp_{m,\mu}$ is a finite sum of the form $\sum_k c_k \,\G^{\lam, m}_{+} [t^k]$, where
 the terms corresponding to $ k\le m-2$ with $m-k$  even are zero. For all other $k$ in this sum
 (we denote this set by  $\K$), we have
\[
(\G^{\lam, m}_{+} [t^k])(r)=\a_{k,m} \,r^{k}, \qquad \a_{k,m}= \frac{1}{c_{\lam,m}}\intl_0^1  (1 - s^2)^{\lambda -1/2}\, C^\lambda_m (s) \, s^k\, ds,\]
where $\lam=(n-2)/2$.  By (\ref{89zse}),
 $\a_{k,m}\neq 0$. Thus, (\ref{zaehle6}) yields
$$\sum\limits_{k\in \K} c_k\,  \a_{k,m}\, r^{k}=0 \quad \forall r>0.$$
It follows that all $c_k$ with $k\in \K$ are zero and $\vp_{m,\mu}(t)$ contains only terms corresponding to $m-k\ge 2$  even.
This completes the proof.
\hfill $\square$

\subsubsection{The Kernel of $R$}

Theorems \ref{zaeh} and \ref{zaehle4} combined with the  formula
\be\label {jikbBU2} (Rf)(\theta, t)\!=\!\frac{\sig_{n-1}}{2|t|}\, (R^* Bf)\left (\frac{\theta}{t}\right ), \qquad (Bf)(\theta, t)\!=\!\frac{1}{|t|^n} \,
f \left (\frac{\theta}{t}\right )\ee
(see Lemma \ref{iozesf}), enable us to  describe  the kernel of the Radon transform $R$. We first prove the following simple lemma.
\begin{lemma} \label{azw1a} If
\be\label{azw1X}  I_1 (f)=\intl_{|x|>a} \!\frac{|f(x)|}{|x|}\, dx<\infty\quad \mbox{for  all $a>0$},\ee
then $Bf\in L^1_{loc} (Z_n)$. If, moreover,
\be\label{azw1} I_2(f)\!=\!\intl_{|x|<a}\! \!\!|x|^{N-1} |f(x)|\, dx\!<\!\infty \quad \mbox{for some $N\!>\!0$ and  $a\!>\!0$},\ee
 then $Bf\in S'(Z_n)$.
\end{lemma}
\begin{proof} Changing variables, for any $a>0$ we have
\[\intl_{\sn} d_*\th \intl_{-1/a}^{1/a}|(Bf)(\theta, t)|\, dt=\frac{2}{\sig_{n-1}}\intl_{|x|>a} \!\frac{|f(x)|}{|x|}\, dx.\]
Similarly,
\[ \intl_{Z_n}\frac{|(Bf)(\theta, t)|}{(1+|t|)^N}\, d_*\th dt=\frac{2}{\sig_{n-1}}\intl_{\rn} \frac{|x|^{N-1}}{(1+|x|)^N}\, |f(x)|\, dx.\]
This gives the result.
\end{proof}

The condition (\ref{azw1}) allows $f(x)$  to  grow as $x\to 0$, but not faster than some power of $|x|^{-1}$. The condition (\ref{azw1X}) is necessary for the existence of the Radon transform on the set of radial functions; cf. Theorem \ref {byvs1}.

Theorems \ref{zaeh} and \ref{zaehle4} in conjunction with Lemma \ref{azw1a} yield the following statements in which
 $f_{m,\mu} (r)$ denote the Fourier-Laplace coefficients  of the function $f_r (\th)=f(r\th)$ and the inequality $f \neq 0$ means  that the set $\{x: f (x)\neq 0\}$ has positive measure.

\begin{theorem} \label{azw1a2}  Let $I_1 (f)<\infty$. Suppose that  $f_{m,\mu} (r)=0$ for almost all $r>0$ if $m=0,1$, and
\be \label{azw2}
f_{m,\mu} (r)= \sum_{\substack{k=0 \\  m-k \,  even }}^{m-2} \frac{c_{k}}{r^{n+k}}, \qquad c_{k}=\const,\ee
if $m\ge 2$. Then $(Rf)(\theta, t)=0$  almost everywhere on $Z_n$.
\end{theorem}

\begin{theorem} \label{zaehle4RA}  Let $I_i (f)\!<\!\infty$; $i\!=\!1,2$.
 Suppose that $(Rf)(\theta, t)\!=\!0$  almost everywhere on $Z_n$.  Then each  Fourier-Laplace coefficient $f_{m,\mu} (r)$ is a finite linear combination of functions $r^{-n-k}$, $k=0,1, \ldots$,  and the following statements hold.

\noindent (i) If $m=0,1$, then $f_{m,\mu} (r)\equiv 0$.

\noindent (ii)  If $m\ge 2$ and $f \neq 0$,  then $f_{m,\mu} (r)\not\equiv 0$  for at least one pair $(m,\mu)$. For every such pair, $f_{m,\mu} (r)$ has the form (\ref{azw2}).
\end{theorem}

{\it Proof of Theorems \ref{azw1a2}   and \ref{zaehle4RA}. }  \ If  $Rf=0$  a.e. on $Z_n$, then $R^* Bf=0$ a.e. on $\rn$. Hence, by  Theorem \ref{zaehle4},  for $t>0$ we have
\[(Bf)_{m,\mu} (t)\!=\!t^{-n} \!\!\intl_{S^{n-1}} \!\!f \left (\frac{\theta}{t}\right )\, Y_{m,\mu} (\th)\, d\th \!=\!\left \{ \begin{array} {ll} 0  \quad &{\rm if} \!\quad m=0,1,\\
\sum\limits^{m-2}_{k=0}{}^{'} c_{k}\, t^{k} &{\rm if}\! \quad m\ge2,
\end{array}\right.\]
where  $\sum^{'}$ includes only those terms for which $m-k$ is even. Changing variable $t=1/r$, we obtain (\ref{azw2}). Conversely, if  $f_{m,\mu} (r)=0$ for  $m=0,1$, and (\ref{azw2}) holds  for  $m\ge 2$, then $(Bf)_{m,\mu} (t)=0$ if $ m=0,1$, and
 $(Bf)_{m,\mu} (t)=\sum\limits^{m-2}_{k=0}{}^{'} c_{k}\, t^{k}$ if $ m\ge2$.
The last equality  is obvious for $t>0$. If $t<0$, then
\bea (Bf)_{m,\mu} (t)&=&\intl_{S^{n-1}} (Bf)(\theta, t)\, Y_{m,\mu} (\th)\, d\th\nonumber\\
&=&(-1)^m \intl_{S^{n-1}} (Bf)(\theta, |t|)\, Y_{m,\mu} (\th)\, d\th\nonumber\\
&=&(-1)^m \sum\limits^{m-2}_{k=0}{}^{'} c_k\, |t|^{k}=\sum\limits^{m-2}_{k=0}{}^{'} c_{k}\, t^{k}\nonumber\eea
because  $m-k$ is even. Hence, by Theorem \ref{zaeh},  $R^* Bf=0$ a.e. on $\rn$ and therefore, by (\ref{jikbBU2}), $Rf=0$ a.e. on $Z_n$.
\hfill $\square$

\begin{example} {\rm Consider the function $f(x)= |x|^{-n} Y_2 (x/|x|)$, $x\neq 0$, where $Y_2$ is a spherical harmonic of degree $2$. This function
 has a non-integrable singularity at the origin  and the integrals of $f$ over hyperplanes through the origin are not absolutely convergent.
 However,  $(Rf)(\th, t)$ is represented by an absolutely convergent integral
for all $ (\th, t)$ with  $t\neq 0$
 and is continuous on the open half-cylinders $C_{\pm}=\{(\th, t)\in Z_n:\, \pm t> 0\}$. Since $f$
 obeys (\ref{azw1}) with any $N>1$, then, by Theorem \ref{azw1a2}, $(Rf)(\th, t)\equiv 0$ in $C_{\pm}$. The latter means that, by continuity, we can also set $(Rf)(\th, t)\equiv 0$   at the points of the form $(\th,0)$, $\th\in S^{n-1}$.
}
\end{example}

\begin{remark}\label {affine in} {\rm We remind the reader that the Radon transform in our treatment is defined assuming that the space $\rn$ has the Euclidean structure. It means that  the origin $(0, \ldots 0)$ is fixed. Theorems \ref{zaeh} and \ref{azw1a2} are formulated in accordance with this structure. Hence, they are not affine invariant.
}
\end{remark}

\subsubsection{Support Theorems}
Theorems \ref{zaeh} and \ref{azw1a2} yield the following  versions of Helgason's support theorem; cf. \cite [p. 10]{H11}. For  $a>0$, we denote
\[ B_a^+ =\{ x\in \bbr^n: |x|<a\}, \qquad   B_a^- =\{ x\in \bbr^n: |x|>a\}, \]
\[  C_a^+ =\{ (\th, t)\in Z_n: |t|<a\}, \qquad C_a^- =\{ (\th, t)\in Z_n: |t|>a\}.\]

\begin{theorem} \label{azw1a2R}  Let $a>0$. If $f(x)=0$ for almost all $x\in B_a^-$, then $(Rf)(\th, t)\!=\!0$ a.e. on  $C_a^-$. Conversely, if \be\label{osllsg} \intl_{B_a^-} |f(x)|\,|x|^m\, dx <\infty\quad \forall \, m\in \bbn\ee
and $(Rf)(\th, t)\!=\!0$ a.e. on  $C_a^-$, then  $f(x)=0$ for almost all $x\in B_a^-$.
\end{theorem}
\begin{proof} The first statement is obvious if $f$ is continuous. In the general case we set $f_a (x)=f (x)$ if $|x|>a$ and $0$ otherwise. By (\ref{duas3}),
\be\label{duas3SZ} \intl_{Z_n} \frac{(R [|f_a|])(\theta,
t)}{(1+t^2)^{n/2}}\,d_*\theta dt= \intl_{\bbr^n}
\frac{|f_a(x)|}{(1+|x|^2)^{1/2}}\,dx. \ee
Since the right-hand side is zero, then so is the left-hand side, and, therefore, $R [|f_a|]=0$ a.e. on $Z_n$. Hence, $Rf=0$ a.e. on $C_a^-$.

Conversely, if  $f$ obeys (\ref{osllsg}), then, by (\ref{duas3}), the integral \[\intl_a^b dt \intl_{\sn} |(Rf)(\th, t)|\, d\th\] is finite for all $b\in (a,\infty)$. Hence,
\be\label {lp45SW} \intl_a^b |(Rf)_{m,\mu} (t)|\, dt\le c \intl_a^b \intl_{\sn} |(Rf)(\th, t)|\, d\th dt<\infty.\ee
If $(Rf)(\th, t)=0$ for almost all $(\th, t)\in C_a^-$, then the right-hand side of (\ref{lp45SW}) equals zero for all $b>a$.  Hence,  the left-hand side is also zero and, therefore,  $(Rf)_{m,\mu} (t)=0$  for almost all $t\notin (-a, a)$. By
Theorem \ref {recon65}, it follows that $f_{m,\mu} (r)=0$  for almost all $r>a$. Invoking the Poisson integral, as in the proof of Theorem
\ref{zaeh}, we conclude that  $f(x)=0$ a.e. whenever $|x|>a$.
\end{proof}

\begin{theorem} \label{azw1a2R2}    Let $\vp$ be an even  function on $Z_n$, $a>0$. If $\vp(\th,t)\!=\!0$ a.e. on $C_a^+$, then   $(R^*\vp)(x)\!=\!0$ a.e. on $B_a^+$. Conversely, if
 \be\label{osllsgdu} \intl_{C_a^+} |\vp(\th,t)|\,|t|^{-m}\, dtd\th  <\infty \quad \forall \,m\in \bbn\ee
and $(R^*\vp)(x)\!=\!0$ a.e. on $B_a^+$, then  $\vp(\th,t)\!=\!0$ a.e. on $C_a^+$.
\end{theorem}
\begin{proof}   The statement  follows from the previous theorem by (\ref{jikb}).
\end{proof}
\begin{remark}\label{his impl} {\rm  The condition (\ref{osllsg}) gives an example of a class of functions for which the implication
\be \label{osllsgdu5}\text{\rm $(Rf)(\th, t)\!=\!0$  on  $C_a^-$} \quad  \Longrightarrow \quad  \text{\rm $f(x)\!=\!0$ on  $B_a^-$} \ee
is true. However, in general, this implication does not hold. For example, every function of the form
\[f(x)=Y_m (x')\sum\limits^{m-2}_{k=0}{}^{'} \; \frac{c_{k}}{|x|^{n+k}},\qquad x'=x/|x|,\]
where  $\sum^{'}$ includes only those terms for which $m-k$ is even, has the vanishing Radon transform on $C_a^-$; cf. Theorem \ref{azw1a2}. On the other hand, a rapid decrease of $f$  is  not necessary  for (\ref{osllsgdu5}), as can be easily seen, by taking functions of the form $f(x)=Y_m (x/|x|) \,g(|x|)$ with $m=0,1$; cf.
Theorem \ref{recon65}.  A similar remark can be addressed to  Theorem  \ref{azw1a2R2}. }
\end{remark}

\section{Spheres through the Origin}\label{Cormack-Quinto}

Below we consider the spherical mean Radon-like  transform which assigns to a function $f$ on  $\rn$ the integrals of $f$ over  spheres passing to the origin. This transform is defined by the formula
\be\label {Corma} (\Q f) (x)=\intl_{\sn} f (x+|x|\,\th)\, d_*\th, \ee
 where $d_*\th$ is the normalized surface element,  so that $\int_{\sn}d_*\th=1$. Thus, $f$ is integrated in (\ref{Corma})
  over the sphere of radius $|x|$ with center at $x$.
 There is a remarkable connection between
 (\ref{Corma}) and  the Darboux equation. Specifically,  in the classical
 Cauchy problem for the Darboux equation we are looking for  a function $u(x,t)$ satisfying
\begin{equation}\label{EPD.1CQ}
\Delta u-u_{tt}-\frac{n-1}{t}\,u_{t}\!
=\!0, \qquad u(x,0)\!=\!f(x),\quad u_{t}(x,0)\!=\!0.
\end{equation}%
Here $x\in \rn$, $t>0$, $f$ is a given function. Now consider the  inverse problem: \textit{ Given   the trace $u\left( x,|x|\right)  $ of  the solution  of  (\ref{EPD.1CQ})  on the cone   $t=|x|$,  reconstruct the initial function $f(x)$. }
It is known that the solution  of the Cauchy problem  (\ref{EPD.1CQ}) has the form
\be\label {mf}
u(x,t)=\intl_{S^{n-1}} f (x+t\theta)\, d_*\theta;\ee
see, e.g.,  \cite [p. 699]{CH}. Hence, the above inverse problem reduces to reconstruction of $f$ from  $(\Q f) (x)$.

The  study of the operator  (\ref{Corma}) relies on the  following connection between $(\Q f) (x)$, the dual Radon
transform (\ref{durt}), and the Radon transform (\ref{rtra1}).
\begin{lemma} \label {CormaQ} Let $n\ge 2$. Then
\be\label {Corma2}
(\Q f)  (x)=  |x|^{2-n}(R^*\vp)(x), \qquad \vp (\th, t)=(2|t|)^{n-2} \,
  f(2t\th),\ee
  and
\be\label {Corma2a}
\!\!\!\!(\Q f)  (x)\!=\!  |x|^{1-n}(R\psi)\left(\frac{x}{|x|}, \frac{1}{|x|} \right), \quad \psi (x)\!=\!\frac{2^{n-1}}{\sig_{n-1}} |x|^{2-2n}
f\left(\frac{2x}{|x|^2} \right),\ee
provided that  either side of the corresponding equality exists in the Lebesgue sense.
\end{lemma}
\begin{proof} The formula (\ref{Corma2}) is due to Cormack and Quinto up to a minor change of notation; cf.  \cite[formula (11)]{CoQ}. To prove it, let $x=r\eta$, $r>0$, $\eta \in \sn$. Then  (\ref{Corma2}) becomes
\be\label {Corma1}
(\Q f)  (x)=2^{n-2}\intl_{\sn} f (2r (\eta \cdot \th)\, \th )\, |\eta \cdot \th|^{n-2} \, d_*\th. \ee
Choose $\gam \in O(n)$  so that $\eta=\gam e_n$. Changing variable $\th=\gam \xi$ and setting $f_{r,\gam} (x)=f(r \gam  x)$, we have
\bea
&&(\Q f)  (x)\!=\!\intl_{\sn} f (r \gam e_n \!+\!r \gam \xi)\, d_*\xi\!=\!\!\intl_{\sn} f_{r,\gam}  ( e_n + \xi)\, d_*\xi\nonumber\\
&&=\frac{1}{\sig_{n-1}}\intl_{-1}^1 (1-t^2)^{(n-3)/2} dt\intl_{S^{n-2}}\!\!f_{r,\gam} (\sqrt {1-t^2} \, \eta +(1+t)\, e_n)\, d\eta.\nonumber\eea
Put $t=2s^2 -1$. This gives
\bea
&&\!\!\!\!\!\!(\Q f)  (x)\!=\! \frac{2^{n-1}}{\sig_{n-1}}\intl_{0}^1 (1\!-\!s^2)^{(n-3)/2}\, s^{n-2}\, ds\!\intl_{S^{n-2}}f_{r,\gam} (2s(\sqrt{1\!-\!s^2} \,\eta \!+\!s e_n))\,d\eta\nonumber\\
&&=2^{n-2} \intl_{\sn}  f_{r,\gam} (2\xi\, (\xi \cdot e_n))\, |\xi \cdot e_n|^{n-2}\, d_*\xi.\nonumber\eea
The last expression coincides with the right-hand side of (\ref{Corma1}).  The equality (\ref{Corma2a})  follows from (\ref{Corma2}) and (\ref{jikb}).
\end{proof}

The following existence result is a consequence of (\ref{Corma2}) and Corollary \ref{gttgzuuuuh} (one can alternatively use (\ref{Corma2a}) and Theorem \ref{byvs1}).
\begin{theorem}\label {llyinteg} If
\be\label {Corma3} \intl_{|x|<a} \frac{|f  (x)|}{|x|}\, dx <\infty\quad \forall \, a>0,\ee
then $(\Q f)(x)$ is finite for almost all $x$. If $f$ is nonnegative, radial, and (\ref{Corma3}) fails  for some $a>0$, then $(\Q f)(x)\equiv \infty$.
\end{theorem}

Since every function in $L^p (\rn)$, $1\le p< n/(n-1)$, can be uniquely reconstructed from its Radon transform, the equality (\ref{Corma2a}) implies the following statement.
\begin{theorem} \label {llyinteg1}If
\be\label{gYYvtCo}
|x|^{2(n-1-n/p)} f(x) \in L^p (\rn), \qquad 1\le p< \frac{n}{n-1},\ee
then $f$ can be uniquely reconstructed from $\Q f$ by the formula
\be\label{gbDDDCo}
f(x)\!=\!2^{n-1}\sig_{n-1} |x|^{2-2n} (R^{-1} g)\left(\frac{2x}{|x|^2} \right), \quad g(\th, t)\!=\!t^{1-n} (\Q f)\left(\frac{\th}{t} \right),\ee
where $R^{-1}$ is the inverse Radon transform.
\end{theorem}

For example, $\Q$ is injective on the class of functions $f$ for which $|x|^{-2} f(x) \in L^1 (\rn)$. It is also injective on  the class
of all compactly supported continuous function on $\rn$. Every such function satisfies (\ref{gYYvtCo}) with $p$ sufficiently close to $n/(n-1)$.

The operator $\Q$ is not injective on the class of {\it all} functions $f$ satisfying (\ref{Corma3}).  The kernel of $\Q$  is described  in the next statement which follows from Theorems \ref{zaeh} and  \ref{zaehle4}.
We recall that the Fourier-Laplace coefficients of $f$ are defined by
\[ \label{zaehl1Qui}  f_{m,\mu} (r)=\intl_{S^{n-1}}f (r\th)\, Y_{m,\mu} (\th)\, d\th, \qquad r>0, \]
 and the inequality $f\neq 0$,  means that  the set $\{x \in \rn: f (x)\neq 0\}$ has positive measure.

\begin{theorem} \label{zaehQuin}
Let $f$  satisfy (\ref{Corma3}).

\noindent (i) \  Suppose that  $f_{m,\mu} (r)=0$ for almost all $r>0$ if $m=0,1$, and
\be \label{azw2INTROQ}
f_{m,\mu} (r)= \sum_{\substack{k=0 \\  m-k \,  even }}^{m-2} c_{k}\,r^{k}, \qquad c_{k}=\const,\ee
if $m\ge 2$. Then  $(\Q f)(x)=0$ for almost all $x\in \rn$.

\noindent (ii) \  Conversely, let $(\Q f)(x)=0$ for almost all $x\in \rn$. Suppose additionally that $f\in S'(\rn)$.
 Then each  Fourier-Laplace coefficient $f_{m,\mu} (r)$ is a finite linear combination of the power functions $r^{k}$, $k=0,1, \ldots$,  and the following statements hold.

\noindent {\rm (a)} If $m=0,1$, then $f_{m,\mu} (r)\equiv 0$.

\noindent {\rm (b)}  If $m\ge 2$ and $f \neq 0$,  then $f_{m,\mu} (r)\not\equiv 0$  for at least one pair $(m,\mu)$. For every such pair, $f_{m,\mu} (r)$ has the form (\ref{azw2INTROQ}).
\end{theorem}

Another consequence of (\ref{Corma2}) is the support theorem  that follows from Theorem \ref{azw1a2R}.  Given $a>0$, we denote by $B_a$ and $B_{2a}$ the balls centered at the origin of radius $a$ and $2a$, respectively.
\begin{theorem} \label{zaehQusu}
 If $f=0$ a.e. in $B_{2a}$, then $\Q f=0$ a.e. in $B_{a}$.  If
\be\label {Corma3su} \intl_{|x|<2a} \frac{|f  (x)|}{|x|^{m+1}}\, dx <\infty \quad  \forall \,  m\in \bbn\ee
 and  $\Q f=0$ a.e. in $B_{a}$, then   $f=0$ a.e. in $B_{2a}$.
\end{theorem}

All these theorems can be  reformulated for the inverse problem (\ref{EPD.1CQ}). For example, the solution to this problem  is unique in the class of  compactly supported continuous functions on $\rn$ and also in the wider class determined by Theorem \ref {llyinteg1}.  Theorem \ref{zaehQusu}  shows that if the trace $u(x, |x|)$ is  zero
for almost all $x\in B_{a}$, then $f(x)= 0$ for  almost all $x\in B_{2a}$  provided that (\ref{Corma3su}) holds.

\section{The Funk Transform}\label{Connection}

 The Funk transform of a function $f$ on the unit sphere $S^n$ in $\bbr^{n+1}$ has the form
\be\label{Funk.aa}
 (Ff)(\th)=\!\!\!
 \intl_{\{  \sig\in S^{n} : \,
\th \cdot \sig =0\}} \!\!\!\! f(\sig) \,d_\th \sig,
\ee
where $d_\th \sig$ stands for the $O(n+1)$-invariant probability measure on the $(n-1)$-sphere $\{\sig \in S^{n}:
\th \cdot \sig =0\}$; see, e.g., \cite{GGG2, H11}. One can readily show  that $Ff$ is well-defined for all $f\in L^1 (S^{n})$ and annihilates odd functions. Below we replenish this statement  using the results of Section  \ref{ndspher} and the link between the Funk transform and the Radon transform.

Let $e_1, \ldots, e_{n+1}$ be the coordinate unit vectors in  $\bbr^{n+1}$,
\[ \bbr^{n-1}= \bbr e_1 \oplus \cdots \oplus \bbr e_{n-1}, \qquad \rn= \bbr e_1 \oplus \cdots \oplus \bbr e_{n},\]
\be\label{r hemisphere} S^{n}_+=\{ \th=(\th_1, \ldots, \th_{n+1})\in S^{n}: 0<\th_{n+1} \le 1\}.\ee
 Consider the projection map
 \be\label{Con22on} \rn \ni x \xrightarrow{\;\mu\;} \th\in S^{n}_+, \qquad \th=\mu (x)=\frac{x+e_{n+1}}{|x+e_{n+1}|}.\ee

\vskip 0.1 truecm
\begin{figure}[h]
\centering
\includegraphics[scale=.7]{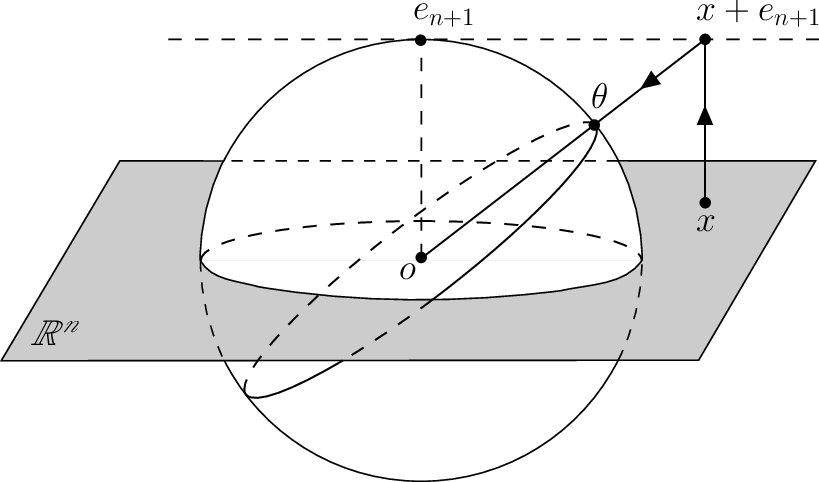}
\caption*{Figure 1: $\mu: \, x \rightarrow \th$.} \label{funk-to-radon}
\end{figure}
\vskip 0.1 truecm

A simple geometric argument  shows that
$ |x|=(1-\th^2_{n+1})^{1/2}/|\th_{n+1}|$ and  the inequalities
$ |x|>a$ and $ |\th_{n+1}|<(1+a^2)^{-1/2}$ are equivalent
 for every $a\ge 0$. Moreover, if $f$ is even, then  (\ref{Con22on}) and  (\ref{hvar}) yield
\be\label{hvaRFr}
 \intl_{|x|>a} \!\!\frac{ f(\mu (x))}{(1+|x|^2)^{(n+1)/2}}\, dx \!=\!\frac{1}{2}\intl_{|\th_{n+1}|<\a} \!\!\!f(\th)\, d\th, \qquad \a\!=\!(1\!+\!a^2)^{-1/2}, \ee
 provided that at least one of these integrals exists in the Lebesgue sense.

 The  map $\mu$ extends to the bijection $\tilde \mu$ from the set $\Pi_n$ of all unoriented hyperplanes in $\rn$ onto the set
    \be\label{Con22on5} \tilde S^{n}_+=\{\om =(\om_1, \ldots, \om_{n+1})\in S^{n}: \,  0\le \om_{n+1}<1\}.\ee
cf. (\ref{r hemisphere}). Specifically, if $\t =\{x\in \rn: x\cdot \eta =t\}\in \Pi_n$, $\eta\in \sn \subset \rn$, $t\ge 0$, and $\tilde \t$ is the $n$-dimensional subspace containing the lifted plane $\t +e_{n+1}$, then $\om \in \tilde S^{n}_+$ is defined to be a normal vector to $\tilde \t$. A simple geometric consideration shows that
 \be\label{Con22on1}  \om =-\eta\, \cos \, \a +e_{n+1}\sin \a, \qquad \tan \a=t.\ee

The above notation is  used in the following theorem.

\begin{theorem} \label{Con22on22} Let $g(x) = (1+|x|^2)^{-n/2} f(\mu (x))$, $x\in \rn$, where $f$ is an even function  on $S^{n}$.  The Funk transform $F$  and the  Radon transform $R$ are related by the formula
\be\label{Con22on2} (Ff)(\om)= \frac{2}{\sig_{n-1}\, \sin d (\om, e_{n+1})}\, (Rg)(\tilde \mu^{-1} \om),  \qquad \om \in \tilde S^{n}_+,\ee
where $d (\om, e_{n+1})$ is the geodesic distance between $\om$ and $e_{n+1}$.
\end{theorem}
\begin{proof} Since the operators on both sides of this equality commute with rotations about the $x_{n+1}$ axis, it suffices to prove the theorem when $\om$ is the $\tilde \mu$-image of the hyperplane
 $\t =\{x\in \rn: x\cdot e_n =t\}$, that is, $\om =-e_n\, \cos \, \a +e_{n+1}\sin \a$, where $\tan \a=t$, $0\le \a<\pi/2$.

Let $\tilde \om=e_n\, \sin \a + e_{n+1} \cos \, \a$. We denote by $r_\om$  a rotation in the $(x_n, x_{n+1})$-plane that takes $e_{n+1}$ to $\tilde \om$. Changing variables and using (\ref{hvar}), we obtain
\bea
(Ff)(\om)&=&\intl_{S^{n} \cap \om^\perp} f(\sig)\, d_\om \sig=\intl_{S^{n-1}} f(r_\om\zeta)\, d_* \zeta\nonumber\\
&=&\frac{2}{\sig_{n-1}}\intl_{\bbr^{n-1}} f(r_\om \, e_y)\,\frac{dy}{|y+e_{n+1}|^n}, \qquad e_y=\frac{y+e_{n+1}}{|y+e_{n+1}|}. \nonumber\eea
Note that
\bea
r_\om \, e_y&=&\frac{y+r_\om e_{n+1}}{\sqrt {1+|y|^2}}=\frac{y+e_n\, \sin \a + e_{n+1} \cos \, \a}{\sqrt {1+|y|^2}}\nonumber\\
&=&\frac{z+e_n \tan \a+ e_{n+1}}{|z+e_n \tan \a+ e_{n+1}|}, \qquad z=\frac{y}{ \cos \, \a}.\nonumber\eea
Hence,
\bea
(Ff)(\om)&=&\!\frac{2}{\sig_{n-1}\, \cos \, \a}\intl_{\bbr^{n-1}} \!\!f \left (\frac{z\!+\!e_n \tan \a\!+\! e_{n+1}}{|z\!+\!e_n \tan \a\!+ \! e_{n+1}|}\right )\,\frac{dz}{(t^2\!+\!|z|^2\! +\!1)^{n/2}}\nonumber\\
&=&\frac{2}{\sig_{n-1}\, \cos \, \a}\intl_{\bbr^{n-1}} f(\mu (z+te_n))\,\frac{dz}{(t^2+|z|^2 +1)^{n/2}}\nonumber\\
&=& \frac{2}{\sig_{n-1}\, \sin d (\om, e_{n+1})}\, (Rg)(e_n,t). \nonumber\eea
This gives the result.
\end{proof}

Theorem \ref{Con22on22} enables us to essentially extend the classes of function $f$ for which the Funk transform $Ff$ is finite a.e. on $S^n$ and is injective. For example, Theorem \ref {byvs1}  and (\ref{Con22on2}) imply the following statement.
\begin{theorem}\label{byvs1SPH}  Let $f$ be an even function on $S^n$. If
\be \label {lkmuxSPH}
\intl_{|\th_{n+1}|<\a} |f(\th)| \,  d\th <\infty \quad \forall \, \a\in (0,1),\ee
 then  $(F f)(\om)$ is finite for  almost all $\om \in S^n$.
If $f$ is nonnegative, zonal, and (\ref{lkmuxSPH}) fails, then  $(F f)(\om)\equiv \infty$.
\end{theorem}
\begin{proof} Following  Theorems \ref {byvs1} and \ref{Con22on22}, we need to transform the integral
\[I=  \intl_{|x|>a} \frac{|g(x)|}{|x|}\, dx= \intl_{|x|>a} \frac{ |f(\mu (x))|}{(1+|x|^2)^{n/2} |x|}\, dx. \]
By  (\ref{hvaRFr}), it can be written as
\[ I=\frac{1}{2}\intl_{|\th_{n+1}|<\a} f_1(\th)\, d\th, \qquad    f_1(\mu (x))= |f(\mu (x))|\, \frac{(1+|x|^2)^{1/2}}{ |x|}\, ,\]
 $\a=(1+a^2)^{-1/2}$. Since
\[
 \intl_{|\th_{n+1}|<\a} f_1(\th)\, d\th= \intl_{|\th_{n+1}|<\a} \frac{ |f(\th)| \, d\th}{(1-\th^2_{n+1})^{1/2}}\le c_\a  \intl_{|\th_{n+1}|<\a} |f(\th)| \, d\th, \]
the results follows.

\end{proof}

Combining Theorem \ref{Con22on22}, (\ref{Con22on1}) and (\ref{hvaRFr}) with  Theorem \ref{azw1a2R}, we arrive at the  support theorem for the Funk transform.

\begin{theorem}\label{SPHSup} For $\a\in (0,1)$, let
\[ \Cal O_\a=\{\th\in S^n:\, |\th_{n+1}|<\a\}, \qquad  \tilde {\Cal O}_\a=\{\om\in S^n:\, |\om_{n+1}|>\sqrt {1-\a^2}\}.\]
If $f=0$ a.e. in $\Cal O_\a$, then $Ff=0$ a.e. in $\tilde {\Cal O}_\a$. Conversely, if
\[ \intl_{\Cal O_\a} |f(\th)| \, |\th_{n+1}|^{-m-1}\, d\th <\infty \quad \forall \, m\in\bbn\]
and $Ff=0$ a.e. in $\tilde {\Cal O}_\a$,   then $f=0$ a.e. in $\Cal O_\a$.
\end{theorem}

In a similar way, Theorems \ref{azw1a2} and \ref{zaehle4RA} yield the corresponding result for the kernel of the operator $F$. We know that $\ker F=\{0\}$ if the action of $F$ is considered on even integrable functions. The situation changes if the functions under consideration  allow  non-integrable singularities at the poles $\pm e_{n+1}$, so that the Funk transform still exists in the a.e. sense.

If $f$ is even, it suffices to consider the points
 $\th\in S^n$  which are represented in the spherical polar coordinates as
\[\th = \eta\, \sin \psi  + e_{n+1} \, \cos \,\psi, \qquad \eta \in \sn, \qquad 0<\psi<\pi/2.\]
The corresponding Fourier-Laplace coefficients (in the $\eta$-variable) have the form
\be\label{TPNY} f_{m,\mu} (\psi)= \intl_{\sn} f(\eta\, \sin \psi  + e_{n+1} \, \cos \,\psi)\, Y_{m,\mu} (\eta)\, d\eta.\ee
 We write $f\neq 0$ if  the set $\{\th \in S^n: f (\th)\neq 0\}$ has positive measure.

\begin{theorem} \label {786NGR1SP}  Let $f$ be an even function on $S^n$ such that
\be\label{TPNY1del} I_1 (f)=\intl_{|\th_{n+1}|<\a} |f(\th)| \,  d\th <\infty\quad \mbox{for  all $\;\a\in (0,1)$}.\ee

\noindent (i) \  Suppose that  $f_{m,\mu} (\psi)=0$ for almost all $\psi \in (0,\pi/2)$ if $m=0,1$, and
\be \label{azw2PHTX}
f_{m,\mu} (\psi)=\sin^{-n} \psi  \sum_{\substack{k=0 \\  m-k \,  even }}^{m-2} c_{k}\,\cot^{k} \psi, \qquad c_{k}=\const,\ee
if $m\ge 2$. Then  $(F f)(\om)=0$  a. e. on $S^n$.

\noindent (ii) \  Conversely, let $(F f)(\om)=0$  a. e. on $S^n$. Suppose, in addition to (\ref{TPNY1del}),  that
\be\label{TPNY1eps}
I_2(f)= \intl_{|\th_{n+1}|>\a} |f(\th)| \,(1-\th^2_{n+1})^{\gam}\,d\th<\infty\ee
for some $\gam>-1/2$ and  $\a\in (0,1)$.
 Then each  Fourier-Laplace coefficient $f_{m,\mu} (\psi)$ is a finite linear combination of the  functions $\sin^{-n} \psi \, \cot^{k} \psi$, $k=0,1, \ldots$,  and the following statements hold.

\noindent {\rm (a)} If $m=0,1$, then $f_{m,\mu} (\psi)\equiv 0$.

\noindent {\rm (b)}  If $m\ge 2$ and $f \neq 0$,  then $f_{m,\mu} (\psi)\not\equiv 0$  for at least one pair $(m,\mu)$. For every such pair, $f_{m,\mu} (\psi)$ has the form (\ref{azw2PHTX}).
\end{theorem}

\begin{proof} By Theorem \ref{Con22on22}, it suffices to reformulate our statement in terms of the function $g(x) = (1+|x|^2)^{-n/2} f(\mu (x))$ and then apply Theorems \ref{azw1a2} and \ref{zaehle4RA}. One can readily check that the assumptions of these theorems (with $f$ replaced by $g$) are equivalent to the corresponding assumptions in Theorem \ref {786NGR1SP} and (\ref{azw2PHTX}) mimics (\ref{azw2}).
\end{proof}

\begin{example} {\rm Let $\{Y_{m,\mu}\}$ be a fixed real-valued orthonormal basis of spherical harmonics in $L^2(\sn)$. Consider any function of the form
\[
f(\th)\!=\!\frac{Y_{2,\mu}(\th'/|\th'|)}{(1\!-\!\th^2_{n+1})^{n/2}}, \quad \th'\!=\!(\th_1, \ldots, \th_n), \quad \mu\!=\!1,2,\ldots, \frac{(n\!+\!2)(n\!-\!1)}{2}.\]
This function is even and satisfies the assumptions of Theorem \ref {786NGR1SP}  for any $\gam>n/2$. Moreover, if $\th= \eta\, \sin \psi  + e_{n+1} \, \cos \,\psi$, $ \eta \in \sn$, $ 0<\psi<\pi/2$, then
 \[ f_{2,\mu} (\psi)=\sin^{-n} \psi
 \intl_{\sn} [ Y_{m,\mu} (\eta)]^2\, d\eta= \sin^{-n} \psi.\]
  Hence, by Theorem \ref{786NGR1SP}, $Ff=0$ a.e. on $S^n$. In fact,
 $(F f)(\om)=0$ for {\it all} $\om$ away from the poles $\pm e_{n+1}$. To see that, it suffices to  smoothen $f$ in arbitrarily small neighborhoods of the poles.
}
\end{example}

\begin{remark}\label{aZSZS3} {\rm As in the Euclidean case, where the origin $(0, \ldots, 0)$ is fixed  (cf.  Remark \ref {affine in}),  here we fix the north pole $(0, \ldots, 0, 1)$. If we choose another point as a pole, the statement of Theorem \ref{786NGR1SP} will be modified accordingly.
}
\end{remark}

\section{The Spherical Slice Transform}\label{Slice}

Let  $S^n$ be the unit sphere in $\bbr^{n+1}$, $n\ge 2$. We denote by
 $\Gam (S^n)$  the set of all $(n-1)$-dimensional geodesic spheres   $\gam \subset S^n$  passing through the north pole $e_{n+1}$.
Every  $\gam$ is a cross-section of $S^n$ by the corresponding hyperplane.
Below we consider an integral transform that assigns to a function $f$ on  $S^n$ a function $\frS f$ on $\Gam (S^n)$
by the formula
\be\label{Sliceint1}
(\frS f) (\gam)=\intl_{\gam} f(\eta)\, d_{\gam}\eta,  \ee
where $d_\gam \eta$ denotes the usual surface element on  $\gam$. The map $f \to \frS f$
is called the {\it spherical slice transform} of $f$.

Every geodesic sphere $\gam \in \Gam (S^n)$  can be indexed  by  its center $\xi=(\xi_1, \ldots, \xi_{n+1})$ in the
the closed  hemisphere
\[
\bar S^n_+=\{\xi=(\xi_1, \ldots \xi_{n+1})\in S^n:\; 0\le \xi_{n+1}\le 1\},\]
so that
 \[\gam \equiv \gam(\xi)=\{\eta \in S^n: \eta \cdot \xi =e_{n+1} \cdot \xi\},    \qquad \xi \in \bar S^n_+.\]
 If $\xi_{n+1}=1$, then $\gam(\xi)$  boils down to one point,  the north pole.  If $\xi_{n+1}=0$, then $\gam(\xi)$ is a ``great circle'' through the poles $\pm e_{n+1}$.

 The operator (\ref{Sliceint1}) has an intimate connection with the  Cauchy problem for the Darboux  equation on $S^n$:
 \be\label {uchypapv}
\del_{\xi} u \!-\! u_{\om\om} \!-\! (n\!-\!1) \cot \om\, u_\om \!=\!0, \qquad u(\xi,0) \!= \!f(\xi), \quad  u_\om (\xi, 0)\! =\! 0.
\ee
 Here $\xi\in S^n$ is the space variable,  $\om\in (0,\pi)$ is the time variable,  $\del_{\xi}$ is the  Beltrami-Laplace operator  acting on  $u(\xi, \om)$  in the $\xi$-variable.  If $(M_\xi f)(t)$  is the spherical mean
 \be \label {uchy77b}
(M_\xi f)(t)=\frac{(1-t^2)^{(1-n)/2}}{\sig_{n-1}}\intl_{\xi\cdot \eta =t} f (\eta)\, d\eta, \qquad t\in (-1,1), \ee
 then the function  $u(\xi, \om)=(M_\xi f)(\cos\, \om)$
 is the solution to the problem (\ref{uchypapv}); see, e.g., \cite {O1, O2}.

The corresponding inverse problem is formulated as follows:

\textit{ Let $d (\xi, e_{n+1})$ be the geodesic distance between the point $\xi$ and the north pole $e_{n+1}$.   Given   the trace $u (\xi, d (\xi, e_{n+1}))$ of  the solution $u(\xi, \om)$  of  (\ref{uchypapv})  on the conical set
 \[\{(\xi,\om): \, \xi \in \bar S^n_+, \; \om =d (\xi, e_{n+1})\},\]
   reconstruct the initial function $f$. }

One can easily see that $u (\xi, d (\xi, e_{n+1}))$ is exactly our slice transform (\ref{Sliceint1}) with $\gam=\gam(\xi)$.

Using spherical coordinates, for $\xi \in \bar S^n_+ $ we  write
\[ \xi= \th \sin \psi +e_{n+1}\, \cos \, \psi, \qquad  \th\in \sn \subset \rn, \quad 0\le \psi \le \pi/2,\]
\[\gam =\gam(\xi)=\gam(\th, \psi), \qquad (\frS f) (\gam)=(\frS f) (\xi)= (\frS f) (\th, \psi).\] Then
 \[\gam(\xi)=\{\eta \in S^n: \eta \cdot \xi =\cos \, \psi\}.\]

Consider the bijective mapping
\be\label{MUIT} \rn \ni x \xrightarrow{\;\nu\;} \eta\in S^{n}\setminus \{e_{n+1}\},  \qquad \nu (x)=\frac{2x+(|x|^2-1)\,e_{n+1}}
{|x|^2+1}.\ee
 The inverse mapping $\nu^{-1}: S^{n}\setminus \{e_{n+1}\} \rightarrow \rn$ is the  stereographic projection  from the north pole $e_{n+1}$ onto $\rn=\bbr e_1 \oplus \cdots \oplus \bbr e_{n}$. If
 \[ \eta= \om \,\sin \vp + e_{n+1} \cos \, \vp, \qquad \om \in \sn, \qquad 0<\vp\le \pi,\]
then $x=\nu^{-1} (\eta)=s\om$, $s=\cot (\vp/2)$;  see Figure 2.

\vskip 0.1 truecm
\begin{figure}[h]
\centering
\includegraphics[scale=.6]{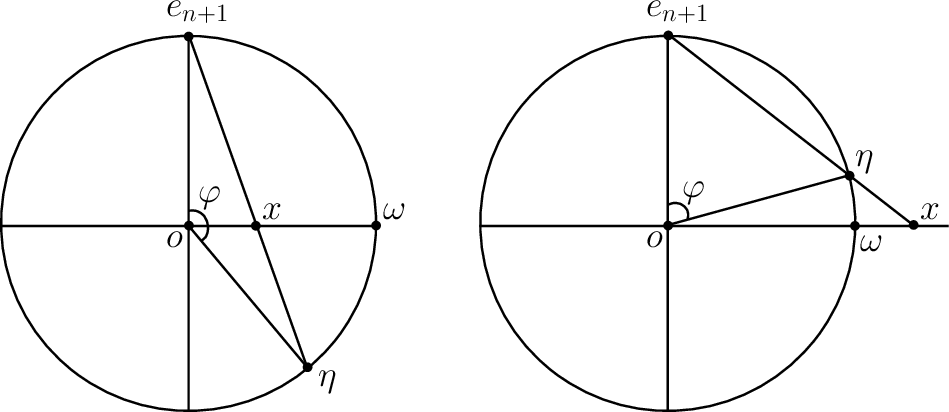}
\caption*{Figure 2: $\eta= \om \,\sin \vp + e_{n+1} \cos \, \vp, \quad |x|=\cot (\vp/2)$.} \label{stereo2}
\end{figure}
\vskip 0.1 truecm

We shall show that the spherical slice transform on $S^n$ can be expressed through the hyperplane Radon transform on $\rn$ by making use of this  projection.

The following statement is a counterpart of Lemma \ref{viat}   and can be found in the literature in different forms; see, e.g., \cite{Mi65}.  For the sake of completeness, we present it  with  a simple proof.

\begin{lemma} \label {stereoviat} ${}$\hfill

\noindent {\rm (i)} If $f\in L^1 (S^n)$, then
\be \label {stslice}\intl_{S^n} f(\eta)\, d\eta= 2^{n}\intl_{\bbr^{n}} (f\circ \nu) (x)\,\frac{dx}{(|x|^2 +1)^{n}}.\ee

\noindent {\rm (ii)}  If $g\in L^1 (\bbr^{n})$, then
\be \label {stslice1} \intl_{\bbr^{n}} g(x)\, dx=\intl_{S^n} (g\circ\nu^{-1}) (\eta)\,\frac{ d\eta}{(1-\eta_{n+1})^{n}}.\ee
\end{lemma}
\begin{proof} {\rm (i)} Passing to spherical coordinates, we have
\bea
l.h.s.
&=&\intl_0^\pi \sin^{n-1} \vp\, d\vp
\intl_{S^{n-1}} \!\!f(\om\, \sin
\vp + e_{n+1}\, \cos \, \vp)\, d\om\nonumber\\
&{}& \qquad \mbox {\rm ($s=\cot (\vp/2)$)}\nonumber\\
&=& 2^n\intl_0^\infty \frac{s^{n-1} ds}{(s^2 +1)^{n}}\intl_{\sn} \!\!\! f \left (\frac{2s\om  \!+\! (s^2\! -\!1)\,e_{n+1}}{s^2 +1}\right )\, d\om\!=\!r.h.s.\nonumber\eea

 {\rm (ii)} We set $g(x)=2^n (|x|^2 +1)^{-n} (f\circ \nu) (x)$ in (\ref{stslice}). Since
  \be\label{ll3567g} |x|^2 +1= s^2 +1=\cot^2 (\vp/2) +1=\frac{2}{1-\cos \vp}=\frac{2}{1-\eta_{n+1}},\ee
  the result  follows.
\end{proof}

\begin{lemma} \label {stereoviatS}
The spherical slice transform on $S^n$ and  the hyperplane Radon transform on $\rn$ are linked by the formula
\be \label {stereogr1r3}
(\frS f) (\theta, \psi)=(Rg)(\theta, t),  \qquad t=\cot \psi,\ee
\be \label {stereogr1r4} g(x)=\left (\frac{2}{|x|^2+1}\right)^{n-1}  (f\circ \nu) (x),\ee
provided that either side of (\ref{stereogr1r3}) is finite when $f$ is replaced by $|f|$.
\end{lemma}
\begin{proof} To prove the lemma, we combine the stereographic projection with translation and rotation.
 Since both $\frS $ and $R$ commute with rotations about the $x_{n+1}$-axis, it suffices to assume $\th=e_n=(0, \ldots, 0,1,0)$. Let $\t_\gam$ be the hyperplane containing  $\gam=\gam(e_n, \psi)$, and let $o' \in \t_\gam$ be the center of the sphere $\gam$. A simple calculation shows that
 \be \label {sVES}
o'=e_n \cos \psi \sin \psi +e_{n+1}\, \cos^2 \psi.\ee
We translate $\t_\gam$ so that $o'$ moves to the origin $o=(0, \ldots, 0)$. Then we rotate the translated plane $\t_\gam -o'$, making it coincident with the coordinate plane $e_n^\perp$ and keeping the  subspace $\bbr^{n-1}=\bbr e_1 \oplus \ldots \oplus \bbr e_n$ fixed.  Let $\tilde\gam \subset e_n^\perp$ be the image of  $\gam$ under this transformation. We stretch $\tilde\gam$ up to the unit sphere $\sn$ in $e_n^\perp$ and
project $\sn$ stereographically onto  $\bbr^{n-1}$; see Figure 3.

\vskip 0.1 truecm
\begin{figure}[h]
\centering
\includegraphics[scale=.7]{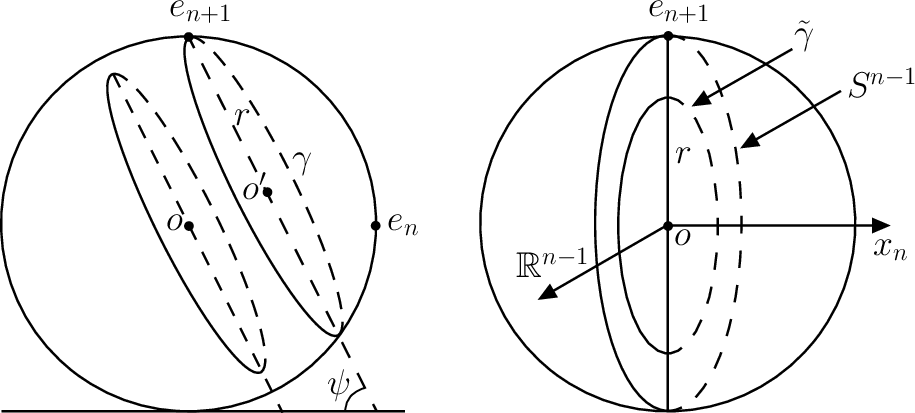}
\caption*{Figure 3: $\gam =o'+\rho \tilde \gam, \; r=\sin \psi$.} \label{stereo2}
\end{figure}
\vskip 0.1 truecm

Thus, can write
\be \label {sVES1} \gam=o'+\rho\tilde\gam, \quad \rho=\left[ \begin{array} {cc} I_{n-1} &  0  \\ 0 & \rho_\psi
\end{array}\right], \quad \rho_\psi=\left[ \begin{array} {cc} \sin \psi &  -\cos \psi  \\ \cos \psi & \sin \psi
\end{array}\right], \ee
\[(\frS f) (e_n, \psi)\!=\!\intl_{\tilde\gam } f(o'\!+\!\rho\sig)\, d_{\tilde\gam }\sig\!=\!r^{n-1}\intl_{\sn}  f(o'\!+\!\rho r\sig)\, d\sig, \quad r\!=\!\sin \psi.\]
By Lemma \ref{stereoviat} (with $n$ replaced by $n-1$), we obtain
\[(\frS f) (\th, \psi)=(2r)^{n-1}\intl_{\bbr^{n-1}} f(o'+\rho r\tilde \nu (y))\, \frac{dy}{(|y|^2 +1)^{n-1}}, \]
 \[\tilde\nu (y)=\frac{2y+(|y|^2-1)\,e_{n+1}}{|y|^2+1}. \]
The expression under the sign of $f$  can be  transformed using (\ref{sVES}) and (\ref{sVES1}):
\[o'\!+\!\rho r\tilde \nu (y)=\frac{A}{|y|^2+1}, \]
\bea A&=&(|y|^2 +1)(e_n \cos \psi \sin \psi +e_{n+1}\, \cos^2 \psi)\nonumber\\
&+&
[2y+(|y|^2\!-\!1)(-e_n \cos \psi +e_{n+1} \sin \psi)]\, \sin \psi\nonumber\\
&=&2y \sin \psi+e_n \sin 2 \psi + e_{n+1} (|y|^2 + \cos 2\psi).\nonumber\eea
Hence,
\bea
&&(\frS f) (e_n, \psi)=(2\sin \psi)^{n-1}\nonumber\\
&&\times \intl_{\bbr^{n-1}} \!\! \!f
\left (\frac{2y \sin \psi+e_n \sin 2 \psi + e_{n+1} (|y|^2 \!+\! \cos 2\psi)}
{|y|^2+1}\right )\,\frac{dy}{(|y|^2 \!+\!1)^{n-1}}. \nonumber\eea
Changing variable $y=u  \sin \psi$ and setting $t=\cot \psi$, this expression can be represented as
\[2^{n-1}\!\!\intl_{\bbr^{n-1}} \!\!\! f\left (\frac{2(u\!+\!te_n)\!+\!(|u\!+\!te_n|^2 \!-\!1) \, e_{n+1}}{|u+te_n|^2 +1}\right)\, \frac{du}{(|u\!+\!te_n|^2 \!+\!1)^{n-1}}.\]
The latter is the Radon transform  $(Rg)(e_n, t)$ with $g$ defined by (\ref{stereogr1r4}).
\end{proof}

Lemma \ref{stereoviatS} enables us to investigate the slice transform $\frS $ using  properties of the hyperplane Radon transform $R$. For example, Theorem \ref{byvs1}, combined with  (\ref{ll3567g}) and Lemma \ref{stereoviat},  gives the following  result.

\begin{theorem} \label {stereoscp2} If
\be \label {stereogr1r50} \intl_{\eta_{n+1}>1-\e}
\frac{|f(\eta)|}{(1-\eta_{n+1})^{1/2}}\,d\eta<\infty \quad \forall \,0<\e\le 2, \ee
 then $(\frS f) (\xi)$ is finite for almost all $\xi \in S^n_+$.
If $f$ is nonnegative, zonal, and (\ref{stereogr1r50}) fails, then  $(\frS f) (\xi)\equiv \infty$.
\end{theorem}

The next statement,  which mimics Theorem \ref {llyinteg1}, is another consequence of  (\ref{ll3567g}) and Lemma \ref{stereoviat}.

\begin{theorem} \label {llhericeg1} If
\be\label{gYYvtCov}
(1-\eta_{n+1})^{n-1-n/p} f(\eta) \in L^p (S^n), \qquad 1\le p< \frac{n}{n-1},\ee
then $f$ can be uniquely reconstructed from $\frS f$ by the formula
\be\label{gbDDDCo}
f(\eta)= (1-\eta_{n+1})^{1-n} (R^{-1} F \circ \nu)(\eta),\ee
where $ F(\th, t)=(\frS f)(\th, \cot^{-1} t)$ and    $R^{-1}$ is the inverse Radon transform.
\end{theorem}

A simple  calculation shows that the injectivity condition (\ref{gYYvtCov}) is stronger than the existence condition (\ref{stereogr1r50}).

\begin{corollary} The operator $\frS$ is injective on the class of functions $f$ for which $(1-\eta_{n+1})^{-1} f(\eta) \in L^1 (S^n)$.
Moreover, it is injective on  $L^\infty (S^n)$.
\end{corollary}
\begin{proof}
The first statement is contained in Theorem \ref{llhericeg1} (set $p=1$). The second one follows from the observation that every bounded function  satisfies (\ref{gYYvtCov}) with $p$  sufficiently close to $n/(n-1)$.
\end{proof}

If $f$ is zonal, then $\frS f$ is zonal too and can be represented by the Erd\'elyi-Kober type fractional integral. To this end, we
 set
\[\eta =\om\, \sin \vp + e_{n+1}\, \cos \, \vp, \qquad \om \in S^{n-1}, \qquad 0<\vp\le \pi.\]
Since $f$ is zonal, then $f(\eta)$ depends only on $\vp$. We denote $f(\eta)=f_0 (\cot \vp/2)$. Similarly, if
 \[ \xi= \th \sin \psi +e_{n+1}\, \cos \, \psi, \qquad \th\in \sn, \qquad 0\le \psi \le \pi/2,\] then $(\frS f)(\xi)$ depends only on $\psi$. We set
  $(\frS f)(\xi)=F_0(\cot \psi)$.

\begin{theorem} \label {stereoscp2} If $f$   is a zonal function satisfying (\ref{stereogr1r50}), then
\be F_0(t)=2^{n-1} \sig_{n-2}\intl^\infty_t \frac{f_0 (r)}{(1+r^2)^{n-1}}  \, (r^2 -t^2)^{(n-3)/2}\, r\, dr.\ee
\end{theorem}
\begin{proof} Since $f$   is  zonal, then  $g$ in  (\ref{stereogr1r3}) is radial. We set  $g(x)= \tilde g (|x|)$. By (\ref{rese}) and  Lemma \ref{stereoviatS},
\[
F_0(t)=\sigma_{n-2} \intl^\infty_{t}\! \tilde g(r)(r^2-t^2)^{(n-3)/2} r \,dr.\]
It remains to express $ \tilde g $ through $f_0$. We have
\[ g(x)=\frac{2^{n-1}\,  (f\circ \nu) (x)}{(1+|x|^2)^{n-1}}, \qquad |x|=|\nu (\eta)|=  \cot \vp/2;\]
cf. (\ref{stereogr1r4})   and (\ref{ll3567g}).  Hence,
\[\tilde g(r)=\frac{2^{n-1}\,  f_0 (r) }{(1+r^2)^{n-1}},\]
and we are done.
\end{proof}

Another  application of the Radon transform theory is related to the spherical harmonic decomposition of
$f(\eta)=f(\om \sin \vp +e_{n+1}\, \cos \, \vp) $ in the $\om$-variable. Let
\[f_{m,\mu} (\vp) =\intl_{\sn }f(\om \sin \vp +e_{n+1}\, \cos \, \vp)\, Y_{m,\mu} (\om)\, d\om.\]
 Then   Theorems \ref{azw1a2} and \ref{zaehle4RA} in conjunction with Lemma \ref{stereoviatS} imply the following  description of the kernel of the operator $\frS$.

\begin{theorem} \label{zasliep}
Let
\be\label{TPNY1delep} I_1 (f)=\intl_{\eta_{n+1}>1-\e}
\frac{|f(\eta)|}{(1-\eta_{n+1})^{1/2}}\,d\eta<\infty \quad \forall \, \e\in (0, 2]. \ee

\noindent (i) \  Suppose that  $ f_{m,\mu} (\vp)=0$ for almost all $\vp \in (0,\pi)$ if $m=0,1$, and
\be\label{zNYUep}
 f_{m,\mu} (\vp)\!=\!(1\!-\!\cos\, \vp)^{1-n} \sum_{\substack{k=0 \\  m-k \,  even }}^{m-2} \!c_{k}\,\left (\tan \frac{\vp}{2}\right )^{n+k}, \quad c_{k}\!=\!\const,\ee
if $m\ge 2$. Then  $(\frS f)(\xi)=0$  a.e. on $S^n$.

\noindent (ii) \  Conversely, let  $(\frS f)(\xi)=0$   a. e. on $S^n$. Suppose, in addition to (\ref{TPNY1delep}),   that
\be\label{TZZ1eps}
I_2(f)=\intl_{\eta_{n+1}<1-\e}
|f(\eta)|\, (1+\eta_{n+1})^{\lam}\,d\eta<\infty\ee
for some $\lam>-1/2$ and  $0<\e\le 2$. Then each  Fourier-Laplace coefficient $f_{m,\mu} (\vp)$ is a finite linear combination of the  functions \[(1-\cos\, \vp)^{1-n} \, \left (\tan \frac{\vp}{2}\right )^{n+k}, \qquad k=0,1, \ldots,\]  and the following statements hold.

\noindent {\rm (a)} If $m=0,1$, then $f_{m,\mu} (\vp)\equiv 0$.

\noindent {\rm (b)}  If $m\ge 2$ and $f \neq 0$,  that is, the set $\{\eta: f (\eta)\neq 0\}$ has positive measure, then $f_{m,\mu} (\vp)\not\equiv 0$  for at least one pair $(m,\mu)$. For every such pair, $f_{m,\mu} (\vp)$ has the form (\ref{zNYUep}).
\end{theorem}

In a similar way,  Theorem \ref{azw1a2R}  implies the  support theorem for the slice transform.
\begin{theorem} \label{zasli} Given $a\in (0,1)$, consider the spherical caps
\[ \Om_a\!=\!\{\eta \in S^n: \, \eta_{n+1} >a\}, \qquad  \tilde \Om_a\!=\!\{\xi \in S^n: \, \xi_{n+1} >\sqrt {(1\!+\!a)/2}\,\}.\]
If $f=0$ a.e. in $\Om_a$, then $\frS f =0$ a.e. in $\tilde \Om_a$.  Conversely, if
\[ \intl_{\Om_a}  (1-\eta_{n+1})^{-1-m/2} \,|f (\eta)|\, d\eta<\infty \quad \forall \, m\in \bbn \]
and $\frS f =0$ a.e. in $\tilde \Om_a$, then $f=0$ a.e. in $\Om_a$.
\end{theorem}

\section{The Totally Geodesic Radon Transform on the Hyperbolic Space}\label {Link wi}
We will be dealing with the hyperboloid model of the $n$-dimensional real hyperbolic space $\hn$ which is described in \cite{GGV}; see also \cite{BR1}.  Let $\bbe^{n, 1}\sim \bbr^{n+1}$, $n\ge 2$, be  the $(n+1)$-dimensional  pseudo-Euclidean   real vector space with the inner product
\be\label {tag 2.1-HYP}[{\bf x}, {\bf y}] = - x_1 y_1 - \ldots -x_n y_n + x_{n +1} y_{n +1}. \ee
The  space $\hn$ is realized as the upper sheet of the two-sheeted hyperboloid    in $\bbe^{n, 1}$, that is,
\[\hn = \{{\bf x}\in \bbe^{n,1} :
\| {\bf x} \|^2 = 1, \ x_{n +1} > 0 \}.\]
The corresponding one-sheeted hyperboloid is defined by
\[\overset {*}{\bbh}{}^n = \{
{\bf x} \in E^{n,1}: \|{\bf x} \|^2 = - 1 \}.\]
Both  $\hn$ and $\overset{*}{\bbh}{}^n$ are orbits of the identity component  $G= SO_0(n,1)$  of the special pseudo-orthogonal group $SO(n,1)$ of linear transformations preserving the bi-linear form $[{\bf x},  {\bf y}]$ and having the determinant $1$.

Unlike the boldfaced ${\bf x}, {\bf y}\in \bbe^{n, 1}$, the usual letters $x,y$ will be used for points in $\hn$. The geodesic distance between $x$ and $y$ is defined by $d(x,y) = \cosh^{-1}[x,y]$. We fix the $G$-invariant measure $dx$ on $\hn$ which is normalized so that
\be\label {tRRRHYP} \intl_{\hn} \!f(x)\, dx= \intl_0^\infty \sh^{n -1} r \,  d r\intl_{\sn} \!\! f(\theta\, \sh r  + e_{n+1}  \ch r) \, d \theta\ee
for every  $f\in L^1 (\hn)$.
The totally geodesic Radon transform of a function $f$ on $\hn$ is defined by the formula
\be\label {tag 3.1-HYP} (\frR f) (\xi) = \int\limits_{\{x\in\hn:\,[x, \xi] = 0\}}
f (x) \, d_\xi x,  \qquad \xi \in \hns, \ee
and represents an  even function on $\hns$.  The corresponding  dual  transform of an even function $\vp$ on $\hns$ has the form
\be\label {tag 3.2-HYP} (\frR^\ast \vp) (x) =
\int\limits_{\{\xi \in  \hns :\,[x, \xi] = 0\}} \!  \!\vp (\xi)  \, d_x \xi,
\qquad x \in \hn. \ee
The measures
$d_\xi x$  and $d_x \xi$ are $G$-images of the corresponding measures on the sets
\[\bbh^{n-1} = \{y\in \hn:\, y_n=0\}, \qquad S^{n -1}=\{\eta \in \hns :\, \eta_{n+1}=0\}.\]
 Specifically, let  $\omega_x$ and $\omega_\xi$ be hyperbolic rotations in $G$  satisfying
\be\label {MMM-HYP} \omega_x :\, e_{n +1} \to  x, \qquad \omega_\xi :\, e_n \to \xi.\ee
If $f_\xi (y)=f (\omega_\xi y)$ and  $\vp_x (\eta)=\vp (\omega_x \eta)$, then the precise meaning of the above integrals is the following:
\be\label {MMM-HYP1}
(\frR f) (\xi) = \intl_{\bbh^{n-1}} f_\xi (y)\, dy, \qquad (\frR^\ast \vp) (x) =\intl_{\sn}\vp_x (\eta) \, d_*\eta.\ee
Both $\frR$ and $\frR^\ast$ are $G$-invariant.

Let $S^{n -1}$ and $S^{n -2}$  be  the unit spheres in the coordinate planes $\bbr^n =\bbr e_1 \oplus  \ldots \oplus \bbr  e_n$ and   $\bbr^{n -1}=\bbr e_1 \oplus \ldots \oplus \bbr e_{n -1}$, respectively. The notation
\be\label {tag 2.7-HYPX} a_r=\left[\begin{array} {ccc} I_{n-1} &0 &0\\
0 &\cosh r &\sinh r\\ 0 &\sinh r &\cosh r\end{array}\right ]\ee
is used for   the corresponding hyperbolic rotation in the plane $(x_n, x_{n+1})$.
Given   $x \in \hn$ and $\xi \in \hns$, we set
\bea\label {tag 3.3-HYP}  x \!&=&\! \theta\, \sh r \! +\! e_{n+1}  \ch r \!=\! \omega_\th a_r e_{n +1}, \quad \,\theta\in S^{n -1}, \;
r \in {\bbr}_+,\\
\label {tag 3.3-HYPa}  \xi \!&=& \! \sigma\,\ch \rho \!+\! e_{n +1} \sh \rho\!=\!\omega_\sigma a_\rho e_n, \qquad \sigma \in S^{n -1}, \;   \rho \in {\bbr}.\eea
Here  $\omega_\theta$ and $ \omega_\sigma \in SO (n)$ are
arbitrary rotations satisfying $\omega_\theta e_n = \theta$, $\omega_\sigma e_n = \sigma$; $a_\rho$ has the same meaning as $a_r$ in (\ref{tag 2.7-HYPX}).

\begin{lemma} \label{Lemma 3.1HYP}   Let  $f_\sigma (x) = f (\omega_\sigma x)$,  $\vp_\theta (\xi) = \vp (\omega_\theta \xi)$.  Then
\bea\label {tag 3.4-HYP} (\frR f)(\xi)\!\! &=& \!\!\int\limits^\infty_0 \sh^{n -2} s  \, ds \\
\!\!&\times&\!\! \int\limits_{S^{n -2}}\!\!
f_\sigma (\omega \,\sh s   + (e_n \, \sh \rho  + e_{n +1} \,\ch \rho )\,\ch s)  \, d \omega, \nonumber\\
\label {tag 3.5-HYP}  (\frR^* \vp)  (x) \!\!\!&=&\!\!\! \int\limits_{S^{n -1}}\!\!\! \vp_\theta (\eta^\prime
+ e_n\, \eta_n \,\ch r + e_{n +1}\, \eta_n\,\sh r ) \, d_* \eta, \eea
$\eta\!=\!(\eta', \eta_n)$, provided that the corresponding integrals exist in the Lebesgue sense.
\end{lemma}
\begin{proof}
  Consider the totally geodesic Radon transform (\ref{tag 3.1-HYP}) and set $x = \omega_\sigma a_\rho y$,
where $\sigma$ and $a_\rho$ are the same as in (\ref{tag 3.3-HYPa}). We have
\bea (\frR f)(\xi)&=&
\int\limits_{\bbh^{n-1}} f_\sigma
(a_\rho y)  \,  dy\nonumber\\
&=& \int\limits^\infty_0 \sh^{n -2} s \
d s \int\limits_{S^{n - 2}} f_\sigma
(a_\rho (\omega\,\sh s   +  e_{n +1}\,\ch s )  \,  d \omega. \nonumber\eea
This  gives (\ref{tag 3.4-HYP}). Further, setting $\xi = \omega_\theta a_r \eta$   in (\ref{tag 3.2-HYP}),
we obtain
\[ (\frR^* \vp)(x)\!=\! \int\limits_{S^{n -1}}\!\!
\vp_\theta (a_r \eta)  \,  d_* \eta \!=\! \int\limits_{S^{n -1}}\!\! \vp_\theta
(\eta^\prime
\!+\!  e_n\,\eta_n\,\ch r \  \!+\! e_{n +1}\,\eta_n \,\sh r) \,  d_* \eta. \]
\end{proof}

The totally geodesic  transform (\ref{tag 3.1-HYP}) and its dual (\ref{tag 3.2-HYP})  are intimately connected with the hyperplane  Radon
transform $R$   and its dual $R^*$. To establish this connection we fix the notation by setting
\be\label {tag 3.10-HYP} (R g) (\sig, t) \!=\! \int\limits_{\sig^\perp} g (\sig t \!+\! u)
 \,  d_\sig u, \qquad (R^\ast h) (y) \!=\!\int\limits_{S^{n -1}} h (\sig, y\cdot \sigma) \,  d_* \sigma. \ee
Here $t \in \bbr$ and  $\sig^\perp$ is the subspace of
$\bbr^n$ orthogonal to $\sig \in S^{n -1}$. Let
\[   x = \theta\, \sh r \! +\! e_{n+1}  \ch r, \qquad \,\theta\in S^{n -1}, \;
r \in {\bbr}_+,\]
\[ \xi = \sigma\,\ch \rho \!+\! e_{n +1} \sh \rho, \qquad \sigma \in S^{n -1}, \;   \rho \in {\bbr}.\]
We also write $\;\tilde x = (x_1, \ldots, x_n),\quad  \tilde \xi = (\xi_1, \ldots, \xi_n)$,
\[x= (\tilde x, x_{n +1}) = (\theta\,\sh r, \ch r)
\in \hn, \]
\[ \xi= (\tilde \xi, \xi_{n +1}) =
(\sigma \,\ch \rho, \sh \rho) \in \hns,\]
 \[f (x) \equiv f (\theta \,\sh r, \ch r), \qquad
\vp (\xi) \equiv
\vp (\sigma\ch \rho, \sh \rho).\]

To every  $x\! =\! (\tilde x, x_{n +1})\in \hn$ we associate its image  $y$ in the tangent hyperplane to $\hn$ at the point  $(0, \ldots, 0,1)\in \hn$,  so that $x$ and $y$ lie on the same line through the origin $(0, \ldots, 0,0)$ of $E^{n,1}$. If this tangent hyperplane is identified with the Euclidean space $\rn$, then  the map $x\to y$ is a bijection between $\hn$ and the unit ball
$B_{n} = \{ y \in \bbr^n: \ |y| < 1 \}$,  so that
\be\label {mmqAAWS} y=\frac{\tilde x}{x_{n +1}}, \qquad x\! =\! (\tilde x, x_{n +1})=\left (\frac{y}{\sqrt{1-| y|^2}}, \frac{1}{\sqrt{1-| y|^2}}\right ).\ee
Under this map, every totally geodesic submanifold of $\hn$ is associated with a chord in $B_{n}$ of the same dimension.
The corresponding functions on $\rn$ and $Z_n=\sn \times \bbr$ are defined by
\be\label {tag 3.11-HYP} g (y)=(1- | y|^2)^{-n / 2}_+ f
\left ( \frac{y}{\sqrt{1-| y|^2}}, \frac{1}{\sqrt{1- |y|^2}}\right ), \qquad  y \in \bbr^n, \ee
\be\label {tag 3.11-HYP} h (\sigma, t)=
(1- t^2)^{- n /2}_+
\vp \left ( \frac{\sigma}{\sqrt{ 1- t^2}},
\frac{t}{\sqrt{1\! -\! t^2}} \right ),   \quad\sigma\in \sn, \;  t \in \bbr, \ee
so that
\be\label {tag 3.12-HYP} f (x) =\frac{g (\tilde x / x_{n +1})}{x^{n}_{n +1} } =
\frac{g (\theta\,\tanh  r)}{\ch^{n}  r}\, , \ee
\be\label {tag 3.13-HYP} \vp (\xi) = (1 + \xi^2_{n +1})^{-n/2} \; \! h
\left( \frac{\tilde \xi}{\sqrt{1 + \xi^2_{n +1}}},
\frac{\xi_{n +1}}{\sqrt{1 + \xi^2_{n +1}}} \right ) =\frac{h (\sigma, \tanh \rho)}{
\ch^{n} \rho}\,. \ee

\begin{lemma} \label {Lemma 3.A6HYP} For every $\del \ge 0$,
\be\label {iKOOUY} \intl_{d (x, e_{n+1})>\del} f (x)\, \frac{dx}{x_{n+1}}= \intl_{{\rm tanh} \del <|y|<1}  g(y)\, dy\ee
provided that either integral exists in the Lebesgue sense.
\end{lemma}
\begin{proof} We have
\[ r.h.s.\!=\!\intl_{{\rm tanh} \del}^1 s^{n-1}\, ds\intl_{\sn} g(\th  s)\, d\th\!=\!\intl_\del^\infty \frac{\tanh^{n-1} r }{\ch^2 r}\, dr\intl_{\sn} \!g(\th \, \tanh r)\, d\th.\]
By (\ref{tRRRHYP})  and (\ref{tag 3.12-HYP}), this expression coincides with the left-hand side.
\end{proof}

\begin{lemma} \label {Lemma 3.6HYP}
 The following equalities hold provided that the integral in either side  exists in the Lebesgue sense:
\be\label {tag 3.14-HYP} (\frR f) (\xi) = \frac{1}{\ch \rho} \,(R g)(\sigma, \tanh \rho), \ee
\be\label {tag 3.15-HYP}  (\frR^*\vp) (x) = \frac{1}{\ch r}\,
(R^\ast h) (\theta\,\tanh  r). \ee
\end{lemma}
\begin{proof}  Let
  $\omega_\theta$, $\omega_\sigma  \!\in  \!SO (n)$ be
arbitrary rotations satisfying $\omega_\theta e_n \!=  \!\theta$, $\omega_\sigma e_n \! = \! \sigma$.  We write $g_\sigma (y) = g (\omega_\sigma y)$.  By (\ref{tag 3.4-HYP}) and  (\ref{tag 3.12-HYP}),
\bea  (\frR f) (\xi) &=& \int\limits^\infty_0 \sh^{n -2} s \
d s \int\limits_{S^{n -2}} g_\sigma
\left( \omega\, \frac{\tanh  s}{\ch \rho}  +e_n \,
\tanh \rho  \right )\,
\frac{d \omega}{(\ch s \, \ch \rho)^n} \nonumber\\
&{}& \mbox {\rm (set $t=\tanh s / \ch \rho$)} \nonumber\\
&=& \int\limits^\infty_0 \frac{t^{n -2} dt}{\ch \rho}
\int\limits_{S^{n -2}} g_\sigma (\omega \, t \! +\!
e_n\,\tanh \rho ) \, d \omega \nonumber\\
&=& \frac1{\ch \rho}
\int\limits_{\bbr^{n-1}} g_\sigma
(v \!+\! e_n\,\tanh \rho) \, d v. \nonumber\eea
The last expression coincides with (\ref{tag 3.14-HYP}).  Let us prove (\ref{tag 3.15-HYP}).  Denoting $h_\theta (\sigma, t) = h (\omega_\theta
\sigma, t)$ and using (\ref{tag 3.5-HYP}) and  (\ref{tag 3.13-HYP}), we have
\bea  &&(\frR^*\vp) (x) \!=\! \frac{1}{\sig_{n-1}}\int\limits^1_{-1}
(1 \!-  \!\eta^2_n)^{(n - 3)/2} d \eta_n \nonumber\\
&&\times
\int\limits_{S^{n -2}} \vp_\theta
\left (\om\, \sqrt{1 \!-\! \eta^2_n} \! +\!e_n  \eta_n\ch r + e_{n +1}\eta_n \, \sh r\right ) \, d \omega \nonumber\eea
or
\bea
&&(\frR^*\vp) (x) \!=\!  \frac{1}{\sig_{n-1}}\int\limits^1_{-1} \frac{(1\!-\!\eta^2_n)^{(n - 3)/2} \,d \eta_n}{(1 \!+\! \eta^2_n \,\sh^2 r)^{n/2}}\nonumber\\
&&\times\int\limits_{S^{n -2}} \! h_\theta \left (\frac{\omega\,\sqrt{1 \!-\!\eta^2_n}
 \!+ \!e_n   \eta_n \ch r}{\sqrt{1 + \zeta^2_n \,\sh^2 r}} ,
\frac{\eta_n \sh r}{\sqrt{1 \!+ \!\eta^2_n \,\sh^2 r }} \right ) d \omega.\nonumber\eea
Setting $\eta_n = \zeta_n / \sqrt{\ch ^2 r - \zeta^2_n \, \sh^2 r} $, we continue
\bea  &&(\frR^*\vp) (x) \!=\!
 \frac{1}{\sig_{n-1}\,\ch r} \int\limits^1_{-1} (1\! - \!\zeta^2_n)^{(n - 3)/2} d\zeta_n \nonumber\\
 &&\times
 \int\limits_{S^{n - 2}}
h_\theta (\omega\, \sqrt{1 \!-\! \zeta^2_n}  \!+ \! e_n \zeta_n, \,\zeta_n \tanh r)
\, d \omega \nonumber\\
&&= \frac{1}{\ch r} \int\limits_{S^{n -1}} h_\theta (\zeta, (\zeta \cdot e_n)\,\tanh r) \, d_* \zeta = \frac{1}{\ch r} \,(R^\ast h) (\theta \,\tanh r). \nonumber\eea
\end{proof}

Lemmas \ref{Lemma 3.A6HYP} and \ref{Lemma 3.6HYP} combined with Theorem \ref{byvs1} yield the following existence result for the totally geodesic transform $\frR$.
\begin{theorem} \label {HYYscp2} If
\be \label {HYYscp25}  \intl_{d (x, e_{n+1})>a} |f (x)|\, \frac{dx}{x_{n+1}}<\infty  \ee
for all $a>0$, then $(\frR f) (\xi)$ is finite for almost all $\xi \in \hns$.
If $f$ is nonnegative, zonal, and (\ref{HYYscp25}) fails for some $a>0$, then  $(\frS f) (\xi)\equiv \infty$.
\end{theorem}

In a similar way, the support theorem  the Radon transform $R$ (see Theorem \ref{azw1a2R}) implies the following statement.
\begin{theorem} \label {786NGR} Let $a>0$   and let $\t_\xi$  denote the totally geodesic submanifold in $\hn$ indexed by $\xi\!\in \!\hns$.  If $f(x)=0$ for almost all $x\!\in \!\hn$ satisfying $d (x, e_{n+1})\!>\!a$, then
$(\frR f )(\xi)\!=\!0$ for almost all $\xi$ satisfying $d(\t_\xi, e_{n+1})>a$. Conversely, if $f$ satisfies  (\ref{HYYscp25})
 and   $(\frR f )(\xi)\!=\!0$ for almost all $\xi$ with $d(\t_\xi, e_{n+1})>a$, then $f(x)=0$ for almost all $x\in \hn$ satisfying $d (x, e_{n+1})>a$.
\end{theorem}

We observe an amazing fact that, unlike the Euclidean case in Theorem \ref{azw1a2R}, the above theorem does not require a rapid decay of $f$ at infinity. This fact was  discovered by Kurusa \cite{Ku94}. The reason is that the function $g$ in (\ref{tag 3.14-HYP}) is supported in the unit ball and therefore, the condition (\ref{osllsg}) holds automatically.
Note also  that the condition  $d(\t_\xi, e_{n+1})>a$ is equivalent to $|\xi_{n+1}|> \sh \,a$, and $d (x, e_{n+1})\!>\!a$ is equivalent to $x_{n+1}> \ch \,a$.

  Theorems \ref{azw1a2} and \ref{zaehle4RA}  give the corresponding result for the kernel of the operator $\fr R$. We write $x\in \hn$ in the hyperbolic polar coordinates as $x = \theta\, \sh r  + e_{n+1} \, \ch r$, $\th \in \sn$, $r>0$, and compute the Fourier-Laplace coefficients
\be\label{TPNYH} f_{m,\mu} (r)= \intl_{\sn} f(\th\, \sh r + e_{n+1} \ch r)\, Y_{m,\mu} (\th)\, d\th.\ee

\begin{theorem} \label {786NGR1}  Let
\be\label{THHHep} I_1 (f)=\intl_{x_{n+1}>1+\del}
|f(x)|\, \frac{dx}{x_{n+1}} <\infty \quad \forall \,\del>0. \ee

\noindent (i) \  Suppose that  $f_{m,\mu} (r)=0$ for almost all $r>0$ if $m=0,1$, and
\be \label{azw2INHY5}
f_{m,\mu} (r)= \sh^{-n} r\sum_{\substack{k=0 \\  m-k \,  even }}^{m-2} c_{k}\, \coth^{k} \psi, \qquad c_{k}=\const,\ee
if $m\ge 2$. Then  $(\frR f)(\xi)=0$  a. e.  on $\hns$.

\noindent (ii) \ Conversely, let $(\frR f)(\xi)=0$ a. e. on $\hns$. Suppose additionally that
\be\label{THH1eps}
I_2(f)=\intl_{x_{n+1}<1+\del}
|f(x)| \,(x_{n+1} -1)^{\lam}\,dx<\infty\ee
for some $\del>0$ and $\lam>-1/2$.
Then each  Fourier-Laplace coefficient $f_{m,\mu} (r)$ is a finite linear combination of the functions $\sh^{-n} r \, \coth^{k} \psi$, $k=0,1, \ldots$,  and the following statements hold.

\noindent {\rm (a)} If $m=0,1$, then $f_{m,\mu} (r)\equiv 0$.

\noindent {\rm (b)}  If $m\ge 2$ and $f \neq 0$, that is, the set $\{x: f (x)\neq 0\}$ has positive measure, then $f_{m,\mu} (r)\not\equiv 0$  for at least one pair $(m,\mu)$. For every such pair, $f_{m,\mu} (r)$ has the form (\ref{azw2INHY5}).
\end{theorem}

Analogues of  Theorems \ref {786NGR} and \ref {786NGR1} for the dual transform $\fr R^*$ can be similarly derived from Theorems \ref{azw1a2R2} and \ref {zaeh}, respectively, using the connection (\ref{tag 3.15-HYP}). The corresponding results are left to the interested reader.

We conclude the paper by the following

\noindent{\bf Open Problem.} Keeping in mind that the Gegenbauer-Chebyshev integrals (and the corresponding Radon transforms) have  unilateral structure, we wonder if the following assumptions caused by the method of the proof can be omitted:

$\bullet$   $\;\vp \in S'(Z_n)$ in Theorem \ref {zaehle4},

$\bullet$   $\;f\in S'(\rn)$  in Theorem \ref {zaehQuin}(ii),

$\bullet$   $\;I_2 (f)<\infty$  in Theorems \ref{zaehle4RA}, \ref{786NGR1SP}(ii), \ref{zasliep}(ii), \ref{786NGR1}(ii).

\vskip 0.3 truecm

\noindent {\bf Acknowledgements.} I am grateful to Jan Boman, Sigurdur Helgason, Ricardo Estrada, Gestur \'Olafsson, and  Yuri Luchko for useful comments and discussion. My special thanks for the figures  in this paper go to  	Emily Ribando-Gros who
was  supported by the NSF VIGRE program.

\end{document}